\documentclass[a4paper,11pt]{amsart}
\usepackage{latexsym}
\usepackage{amsmath}

\usepackage{epsfig}

\newtheorem{theorem}{Theorem}
\newtheorem{corollary}[theorem]{Corollary}
\newtheorem{lemma}[theorem]{Lemma}
\newtheorem{proposition}[theorem]{Proposition}

\newcommand{\ZZ}{{\rm\bf Z}}
\newcommand{\OO}{{\rm\bf O}}
\newcommand{\RR}{{\rm\bf R}}

\newcommand{\XX}{{\rm\bf X}}
\newcommand{\EU}{{\rm\bf S}}

\newcommand{\vv}{{\rm\bf v}}
\newcommand{\ww}{{\rm\bf w}}

\newcommand{\dpt}{\displaystyle}

\title[Spiralling Dynamics]{Spiralling dynamics near  heteroclinic networks}

\date{\today}

\author[Alexandre A. P. Rodrigues \and
Isabel S. Labouriau]{Alexandre A. P. Rodrigues \and
Isabel S. Labouriau}
\address{ Centro de Matem\'atica
da Universidade do Porto \\ and 
Faculdade de
Ci\^encias, Universidade do Porto \\
Rua do Campo Alegre,
687, 4169-007 Porto, Portugal}
\thanks{ CMUP is supported by the European Regional Development Fund through the programme COMPETE and by the
Portuguese Government through the Funda\c{c}\~ao para a Ci\^encia e a Tecnologia (FCT) under the project PEst-C/MAT/UI0144/2011.
A.A.P. Rodrigues was supported by the grants SFRH/BD/28936/2006 and SFRH/BPD/84709/2012 of FCT}

\email[A.A.P. Rodrigues]{  alexandre.rodrigues@fc.up.pt}
\email[I.S. Labouriau]{ islabour@fc.up.pt} 

\begin{document}

\begin{abstract}
There are few explicit examples in the  literature  of vector fields exhibiting complex dynamics that may be proved analytically.
We construct explicitly a  {two parameter family of vector fields}  on the three-dimensional sphere $\EU^3$, whose flow has a spiralling attractor containing the following: two hyperbolic equilibria, heteroclinic trajectories connecting them  {transversely} and a non-trivial  hyperbolic, invariant and transitive set. The spiralling set unfolds a heteroclinic network between two symmetric saddle-foci and contains a sequence of topological horseshoes semiconjugate to full shifts over an alphabet with more and more symbols, {coexisting with Newhouse phenonema}. The vector field is the restriction to $\EU^3$ of a polynomial vector field in $\RR^4$. In this article, we also identify global bifurcations that induce chaotic dynamics of different types.
\end{abstract}

\maketitle

\bigbreak
\bigbreak
\bigbreak

\bigbreak
\bigbreak

\bigbreak
\bigbreak
\bigbreak
\bigbreak

Address of Alexandre A. P. Rodrigues \and
Isabel S. Labouriau:\\
{ Centro de Matem\'atica
da Universidade do Porto \\ and 
Faculdade de
Ci\^encias, Universidade do Porto \\
Rua do Campo Alegre,
687, 4169-007 Porto, Portugal}\\
Phone (+351) 220 402 211   
and (+351) 220 402 248   Fax (+351) 220 402 209

\bigbreak

\newpage
\section{Introduction}

Any $C^1$  vector field defined on a compact three-dimensional manifold may be approximated by a system of differential equations whose flow exhibits one of the following phenomena: uniform hyperbolicity, a heteroclinic cycle associated to, at least, one equilibrium and/or a homoclinic tangency of the invariant manifolds of the periodic solutions --- see Arroyo \emph{et al} \cite{Arroyo}. 
However, there are no examples of vector fields exhibiting all three features simultaneously.

In this article, we construct an explicit example of a $C^\infty$ vector field on the three-dimensional sphere $\textbf{S}^3$ that is approximated by differential equations exhibiting all the three behaviours. Nearby differential equations may also display heteroclinic tangencies of invariant manifolds of two equilibria. We describe and characterise some properties of the flow of these differential equations, whose complex geometry arises from spiralling behaviour induced by the presence of saddle-foci. 
{ The complex nature of the geometry can be described analytically since our example is close to a highly symmetric vector field}
{that, by construction, exhibits special features.} 
We start with a brief discussion of the literature on examples of this kind.

\subsection{Lorenz-like attractors}
Few explicit examples of vector fields are known, whose flow contains non-hyperbolic invariant sets that are transitive and for which transitivity is robust to small $C^1$  perturbations.
The most famous example is the expanding \emph{butterfly} proposed by E. Lorenz in 1963 \cite{Lorenz} that arises in a flow having an equilibrium at the origin, where
the linearisation has eigenvalues $\lambda_u, \lambda_s^1, \lambda_s^2 \in \RR \backslash \{0\}$  satisfying:
\begin{equation}
\label{eigenvalues_exp}
\lambda_s^2<\lambda_s^1<0<\lambda_u \quad \text{and} \quad \lambda_u+\lambda_s^1>0.
\end{equation}
In a three-dimensional manifold, an equilibrium whose linearisation has real eigenvalues satisfying condition (\ref{eigenvalues_exp}) is what we call an \emph{equilibrium of Lorenz-type}.

In order to understand the Lorenz differential equations and the phenomenon of robust coexistence in the same transitive set of an equilibrium and regular trajectories accumulating on it, geometric models have been constructed independently by Afraimovich \emph{et al} \cite{ABS} and Guckenheimer and Williams \cite{GW}. The construction of these models have been based on properties suggested by numerics. 

Morales \emph{et al} \cite{MPP} unified the theory of uniformly hyperbolic dynamics and Lorenz-like flows stressing that the relevant notion for the general theory of robustly transitive sets is the \emph{dominated splitting}. More precisely, they have proved that if $\Lambda$ is a robustly transitive attractor containing at least one equilibrium of Lorenz-type, then $\Lambda$ must be partially hyperbolic with volume expanding directions, up to reversion of time.

In order to { make the discussion more rigorous,}
recall that a compact flow-invariant set $\Lambda$ is \emph{partially hyperbolic} if, {up to time reversal}, there is an invariant splitting $T \Lambda= E^s \oplus E^{cu}$ for which there are $K, \lambda \in \textbf{R}^+$ such that for $\forall t>0, \forall x \in \Lambda$:
\begin{itemize}
\item $||\partial_x \phi(t,x)|_{E^s_x}|| \leq K e^{-\lambda t}$; 
\item $||\partial_x \phi(t,x) |_{E^s_x} ||.||\partial_x \phi(t,x)|_{E^{cu}_{\phi(t,x)}}|| \leq  K e^{-\lambda t}$,
\end{itemize}
{where $\phi(t,p)$ is the unique solution $x(t)$ of the initial value problem $\dot{x}=f(x)$, $x(0)=p,$ and $f: \textbf{R}^4 \rightarrow \textbf{R}^4$ is a smooth vector field.}
The direction $E^{cu}$ of $\Lambda$ is \emph{volume expanding} if 
$$
\forall t>0, \quad   \forall x \in \Lambda,\quad  det |\partial_x \phi(t,x)|_{E_x^{cu}}| \geq K e^{\lambda t}.
$$
 The Lorenz model satisfies these conditions.  
A good explanation about this subject may be found in Ara\'ujo and Pac\'ifico \cite[Chapter 3]{AP}, where it is shown that in a three-dimensional manifold the only equilibria that exist near a robustly transitive set must be of Lorenz-type \emph{ie} where condition (\ref{eigenvalues_exp}) holds --- see also Bautista \cite{Bautista}.

\subsection{Contracting Lorenz Models}
In 1981, Arneodo, Coullet and Tresser \cite{ACT1} {started the study of} \emph{contracting Lorenz models}, {whose flows contain attractors that persist only in a measure theoritical sense.}  The authors considered a variation of the classical Lorenz model with respect to the eigenvalues at the origin, in which the condition (\ref{eigenvalues_exp}) is replaced by:
\begin{equation}
\label{eigenvalues_cont}
\lambda_s^2<\lambda_s^1<0<\lambda_u \quad \text{and} \quad \lambda_u+\lambda_s^1<0.
\end{equation}

In 1993, Rovella \cite{Rovella} proved that there exists a \emph{contracting Lorenz model} $\dot{x}=f(x)$, $x \in \textbf{R}^3$ with an attractor $\Lambda$ containing an equilibrium so that the following hold: there exists a local basin of attraction $\mathcal{B}$ of $\Lambda$, a neighbourhood $U$ of $f$ (in the $C^3$-- topology) and an open and dense subset $U_0 \subset U$ so that for $\dot{x}=g(x)$ ($g \in U_0$), the maximal invariant set in $\mathcal{B}$ consists of the equilibria, one or two periodic trajectories, a hyperbolic suspended horseshoe and heteroclinic connections. Moreover, Rovella \cite{Rovella} proved that in generic two parameter families $\dot{x}=f(x, \mu)$, with $f(\star, \overline{0})\equiv f$, there is a set of positive measure containing $\mu=(0,0)$ for which an attractor in $\mathcal{B}$ containing the equilibrium exists.

The construction of the Rovella attractor \cite{Rovella} is similar to the geometric Lorenz model. Some authors constructed contracting  Lorenz-like examples through bifurcations from heteroclinic cycles and networks: for instance, Afraimovich \emph{et al} \cite{ACLiu} describe a codimension 1 bifurcation leading from Morse-Smale flows to Lorenz-like attractors; Morales \cite{Morales} constructed a singular attractor from a hyperbolic flow, through a saddle-node bifurcation. All of these are similar to the expanding Lorenz attractor, for which condition (\ref{eigenvalues_exp}) hold, as opposed to the contracting case.

\subsection{Spiralling attractors}
Another type of persistent attractor with a much more complex geometry occurs near homoclinic and heteroclinic cycles associated to either a saddle-focus or a non-trivial periodic solution.
{This} complexity is the reason why the study of this subject was left almost untouched for 20 years from L. P. Shilnikov  \cite{Shilnikov65, Shilnikov68} 
 to P. Holmes \cite{Holmes}. Due to the existence of complex eigenvalues of the linearisations at the equilibria, the spiral structure of the non-wandering set predicted by Arneodo, Coullet and Tresser \cite{ACT2} has been observed in some simulations in the context of the modified Chua's circuit --- these attractors are what Aziz-Alaoui \cite{Aziz} call \emph{spiralling attractors}. 

Spiralling attractors {are} expected for perturbations of {homo and heteroclinic} cycles involving saddle-foci under a dissipative condition.
In Kokubu and Roussarie \cite{KRoussarie}, the classic Lorenz model is considered as a particular case of a model whose flow contains a heteroclinic cycle.
In this article, the authors carefully choose a perturbation of the classic Lorenz model\
 and have found other types of chaotic dynamics such as H\'enon-like chaotic attractors and Lorenz attractors with hooks. There is some evidence that in this case spiral attractors exist and that they might be persistent in the sense of measure \cite{ACT1, ACT2}. A sequence of topological horseshoes semiconjugate to full shifts on an alphabet with {more and more symbols} might occur near this kind of attractors.

Nowadays, particular attention is being given to the study of the dynamics near heteroclinic networks with complex behaviour and multispiral attractors. 
Spiralling dynamics near a large heteroclinic network of rotating nodes has been studied by Aguiar \emph{et al} \cite{ACL BIF CHAOS, ACL NONLINEARITY}, motivated by a conjecture of Field \cite{Field}.
A symmetry reduction argument yields a {quotient} network with two saddle-foci of different Morse indices reminiscent of those studied by Bykov \cite{Bykov93, Bykov99}. In \cite{ACL NONLINEARITY}, the existence of suspended horseshoes with the shape of
the network has also been proved.

In a similar context and under the assumption that near the saddle-foci solutions wind in the same direction around the one-dimensional heteroclinic connection, Labouriau and Rodrigues \cite{LR} proved the emergence of an intricate behaviour arising in a specific type of symmetry-breaking; no explicit examples have been given.
This creates an interest in the construction of explicit vector fields whose flows have a specific type of invariant set and for which it is possible to give an analytical proof of the properties that guarantee the existence of complex behaviour.

In this article, we describe an explicit example of a polynomial vector field on the three-dimensional sphere $\EU^3$ whose flow has a heteroclinic network with two saddle-foci, and a spiralling structure containing a hyperbolic non-trivial transitive set {enclosing periodic solutions with different stability indices --- this is what Gonchenko \emph{et al} \cite{GST} call a quasistochastic attractor}. We show that the spiralling set is not robustly transitive, but it presents some similarities to { Rovella's example \cite{Rovella},}
 {in the sense that it might be}  {measure-theoretically} {persistent.}

\subsection{Heteroclinic terminology}

Throughout this paper, by \emph{heteroclinic cycle} we mean a set of finitely many disjoint hyperbolic equilibria $p_j$ (also called \emph{nodes}), $j \in \{1,\ldots, k\}$ and trajectories $\xi_j$, $j \in \{1,\ldots, m\}$ such that:
$$
\lim_{t \rightarrow +\infty} \xi_j(t)= p_{j+1}= \lim_{t \rightarrow -\infty} \xi_{j+1}(t),
$$
called \emph{ heteroclinic trajectories}, with the understanding that $\xi_{m+1}=\xi_1$ and $p_{k+1}=p_1$. We also allow connected $n$--dimensional 
manifolds of solutions biasymptotic to the nodes $p_i$ and $p_j$, in negative and positive time, respectively. 
 In both cases we denote the connection by $[p_i\to p_j]$. Throughout this article, the nodes are hyperbolic; the dimension of the local unstable manifold of an equilibrium $p$ will be called \emph{the Morse index} of $p$.

A \emph{heteroclinic network} is a connected set that is the union of heteroclinic cycles, where in particular, for any pair of saddles in the network, there is a sequence of heteroclinic connections that links them.
Heteroclinic cycles arise robustly in differential equations that are equivariant under the action of a group of symmetries as a connected
component of the group orbit of a heteroclinic cycle.  We refer the reader to Golubitsky and Stewart \cite{GS} for more information about heteroclinic cycles and symmetry in differential equations.

A  \emph{Bykov cycle} is a heteroclinic cycle with two saddle-foci of different Morse indices, in which the one-dimensional invariant manifolds coincide and the two-dimensional invariant manifolds have an isolated transverse intersection. 

\section{The vector field and a framework of the article}
\label{Sect2}
\subsection{Motivation}
It is difficult to find explicit examples for which one can prove that Bykov cycles are present, although these cycles are unavoidable features in one-parameter families of vector fields in three dimensions.
Working with vector fields with symmetry simplifies this task in several ways.
The most obvious one is that symmetries imply the existence of flow-invariant submanifolds, on which it is easier to find the fragile connection of one-dimensional manifolds.
In a two-dimensional flow-invariant set, this may appear as a saddle-to-sink connection, that persists under perturbations that preserve the symmetry.
Examples with simple polynomial forms of low degree are natural in symmetric contexts.
A simple polynomial form makes computations easier and allows us to prove the transverse intersection of two-dimensional invariant manifolds.

Vector fields that are close to more symmetric ones are easier to analyse, as is done in \cite{ACL BIF CHAOS, ACL NONLINEARITY,LR}.
In order to illustrate the symmetry breaking reported in \cite{LR}, we  start by defining an organizing center with spherical symmetry whose flow contains an attracting network lying on the unit sphere $\EU^3\subset \RR^4$. By gradually breaking the symmetry in a two-parameter family, we obtain a wide range of dynamical phenomena going from sinks to a spiralling structure of suspended horseshoes and cascades of saddle-node bifurcations near a Shilnikov cycle. The construction is amenable to the analytic proof of the features that guarantee the existence of complex behaviour. More precisely, the { symmetry-breaking}
terms have been chosen in such a way that:
\begin{itemize}
\item the unit sphere $\EU^3$ remains flow-invariant and globally attracting;
\item the cycles persist (although they change their nature);
\item the symmetries are broken gradually and independently;
\item  it is possible to show analytically that the two-dimensional invariant manifolds of the nodes meet  transversely .
\end{itemize}

We are interested in dynamics on a compact manifold, in order to have control of the long-time existence and behaviour of solutions.

\subsection{The vector field}
Our object of study is the  two parameter family of vector fields 
$\XX(X,\lambda_1,\lambda_2)$ on the unit sphere $\EU^3\subset\RR^{4}$, defined for 
$X=(x_1,x_2,x_3,x_4)\in\EU^3$ by  the differential equation in $\RR^{4}$:
\begin{equation}
\left\{
\begin{array}{l}
\dot{x}_{1}=x_{1}(1-r^2)-x_2-\alpha_1x_1x_4+\alpha_2x_1x_4^2+\lambda_2 x_3^2x_4 \\
\dot{x}_{2}=x_{2}(1-r^2)+x_1-\alpha_1x_2x_4+\alpha_2x_2x_4^2 \\
\dot{x}_{3}=x_{3}(1-r^2)+\alpha_1x_3x_4+\alpha_2x_3x_4^2+\lambda_1 x_1x_2x_4 -\lambda_2 x_1x_3x_4 \\
\dot{x}_{4}=x_{4}(1-r^2)-\alpha_1(x_3^2-x_1^2-x_2^2)-\alpha_2x_4(x_1^2+x_2^2+x_3^2)-\lambda_1 x_1x_2x_3 \\
\end{array}
\label{example}
\right.
\end{equation}
where  $r^2=x_{1}^{2}+x_{2}^{2}+x_{3}^{2}+x_{4}^{2}$ and {$\alpha_2<0<\alpha_1$} with $\alpha_1+\alpha_2>0$.

There are two equilibria given by:
$$
\vv =(0,0,0,+1) \quad \text{and} \quad \ww =(0,0,0,-1)
$$
and the linearisation of $\XX(X,\lambda_1,\lambda_2)$ at $(0,0,0,\varepsilon)$ with 
$\varepsilon=\pm 1$ has eigenvalues $$\alpha_2-\varepsilon\alpha_1\pm i \quad \text{and} \quad
\alpha_2+\varepsilon\alpha_1.$$
Then, under the conditions above, $\vv$ and $\ww$ are hyperbolic saddle-foci,
$\vv$ has one-dimensional unstable  manifold and two-dimensional stable manifold; $\ww$ has one-dimensional stable manifold and two-dimensional unstable manifold.
The only other equilibrium of \eqref{example} is the origin.

\begin{figure}
\begin{center}
\includegraphics[height=3cm]{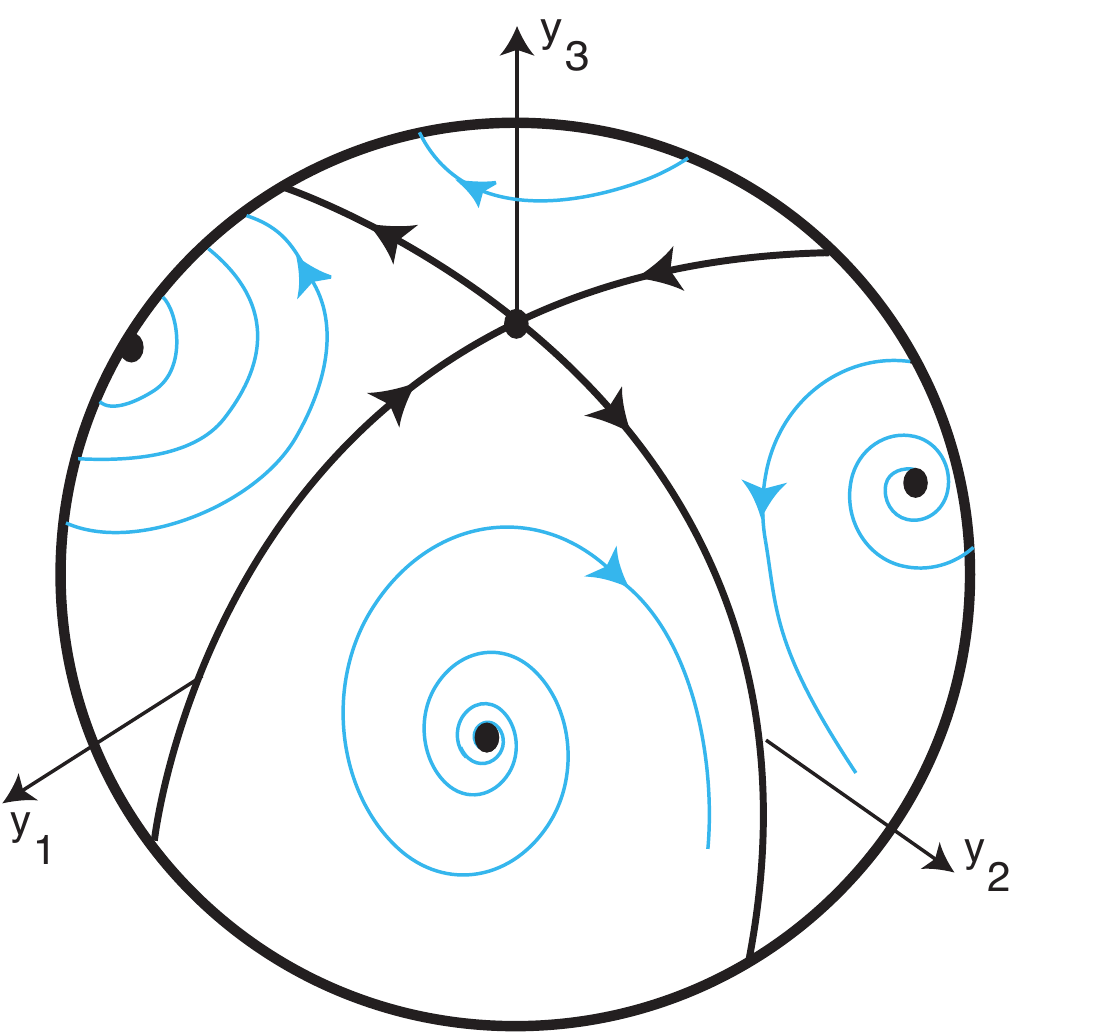}
\end{center}
\caption{\small Dynamics of the intermediate three-dimensional step in the construction of the organising centre $\XX_0(X)=\XX(X,0,0)$. }
\label{Flow2}
\end{figure}

Equation \eqref{example} was obtained using a general construction described in  Aguiar \emph{et al} 
\cite{ACL BIF CHAOS}, that we proceed to summarise in the particular case used here.
Start with the differential equation $\dot Y=(1-|Y^2|) Y$ for  $Y=(y_1,y_2,y_3)\in\RR^3$, for which the unit sphere $\EU^2$ is  globally attracting and all its points are equilibria.
Then consider a finite subgroup $G\subset \OO(3)$ containing 
$$
d(y_1,y_2,y_3)=(y_1,-y_2,y_3)\quad
\mbox{and}\quad
\kappa(y_1,y_2,y_3)=(-y_1,y_2,y_3).
$$
%
The second step is to add two $G$-symmetric perturbing terms, $\alpha_1A(Y)$ and
$\alpha_2B(Y)$, to the differential equation.
The two terms are chosen to be tangent to $\EU^2$, so this sphere is still flow-invariant.
From the symmetry, it follows that the two planes $\{d(Y)=Y\}$ and $\{\kappa(Y)=Y\}$ are also flow-invariant
and the dynamics is that of Figure~\ref{Flow2}.

Now, add the equation $\dot\theta=1$ and interpret $(\theta,y_1)$ as polar coordinates in $\RR^2$.
In the new variables $(x_1,x_2,x_3,x_4)=(y_1\cos\theta,y_1\sin\theta,y_2,y_3)$ this is equation  \eqref{example} for $\lambda_1=\lambda_2=0$,
so, by construction, the unit sphere $\EU^3$ is  invariant under the flow of \eqref{example} for $\lambda_1=\lambda_2=0$ and every trajectory with nonzero initial condition is asymptotic to it in forward time.
Hence, \eqref{example} defines a family of  vector fields  
$\XX(X,0,0)$ on $\EU^3$. 
The invariant  circle  $\{d(Y)=Y\}$ gives rise to an invariant two-sphere  and $\{\kappa(Y)=Y\}$ gives rise to an invariant circle.

Other properties of  \eqref{example} with $\lambda_1=\lambda_2=0$ that follow by construction, are given 
Aguiar \emph{et al} \cite[Theorem 10]{ACL BIF CHAOS}, we describe them briefly.


The group of symmetries of $\XX_0(X)=\XX(X,0,0)$ is isomorphic to $\textbf{SO(2)} \oplus \ZZ_2(\gamma_2)$ where $\psi_\theta \in \textbf{SO(2)}$ acts as:
$$
\psi_\theta(x_1, x_2,x_3,x_4)=(x_1 \cos\theta-x_2 \sin\theta,x_1 \sin\theta+x_2 \cos\theta, x_3,x_4)
$$
and $\gamma_2 \in \ZZ_2(\gamma_2)$ acts as:
$$
\gamma_2(x_1, x_2,x_3,x_4)=(x_1, x_2,-x_3,x_4).
$$
This comes from the symmetry group $G$, plus the rotation into $\RR^4$.

The one-dimensional invariant manifolds of $\vv$ and $\ww$ lie inside the invariant circle $Fix(\textbf{SO(2)}) \cap \EU^3$ and the two-dimensional  invariant manifolds  lie in the invariant two-sphere $Fix(\ZZ_2(  \gamma_2 ))\cap \EU^3$. 
Thus, symmetry forces the invariant manifolds of $\vv$ and $\ww$ to be in a very special position: they coincide, see Figure~\ref{orientation_exemplo}.
The two saddle-foci, together with their invariant manifolds form a heteroclinic network 
$\Sigma^0$ that is asymptotically stable by the criteria of Krupa and Melbourne \cite{Krupa e Melbourne 1, Krupa e Melbourne}.
Indeed since $\alpha_2<0<\alpha_1$, it follows that:
$$
\rho=\frac{C_\vv}{E_\vv}\frac{C_\ww}{E_\ww}=\left(\frac{\alpha_2-\alpha_1}{\alpha_2+\alpha_1}\right)^2>1,
$$
where $E_p$ and $C_p$ denote the real parts of the expanding and contracting eigenvalues of $DX_0$ at
{ $p$, where $p\in \{\vv,\ww\}$}.

\begin{figure}
\begin{center}
\includegraphics[height=7cm]{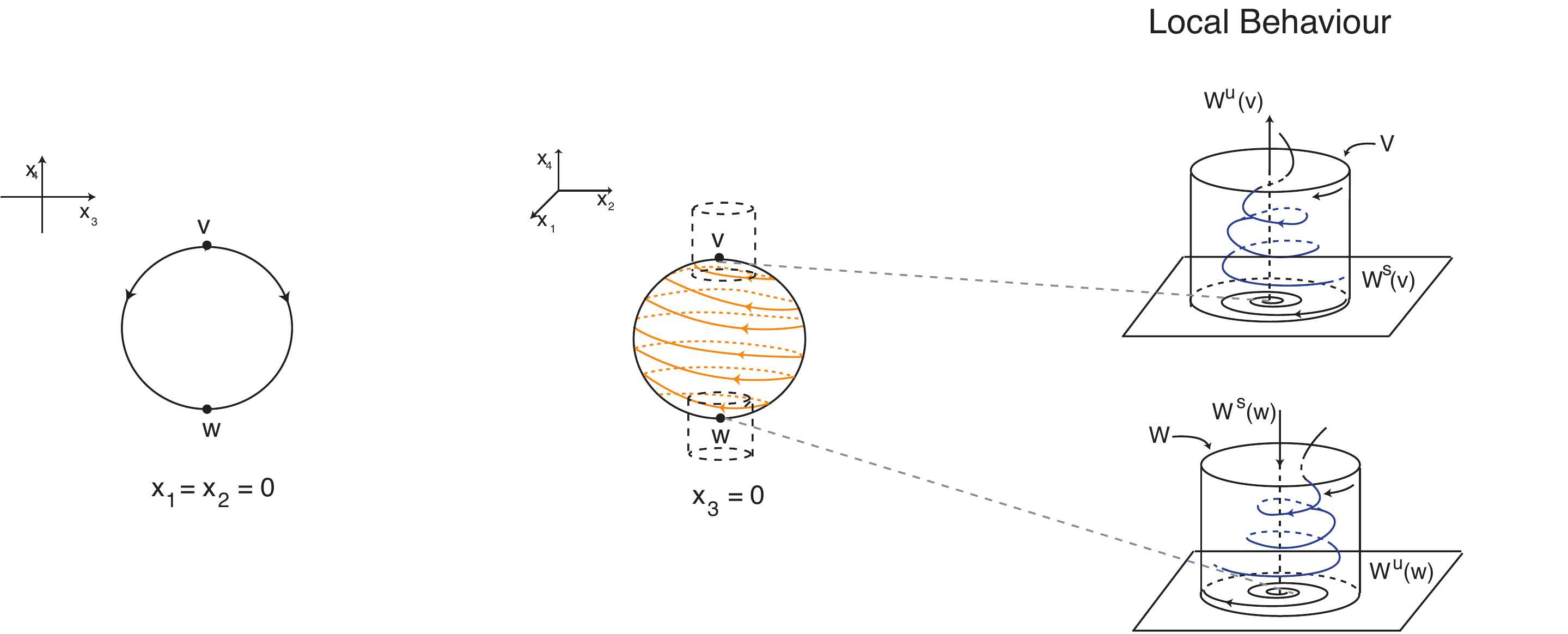}
\end{center}
\caption{\small Dynamics of the organising centre $\XX_0(X)=\XX(X,0,0)$. Left: The one-dimensional heteroclinic connection from $\vv$ to $\ww$ on the invariant circle $Fix(\textbf{SO(2)})\cap\EU^3$. 
Centre: The two-dimensional heteroclinic connection from $\ww$ to $\vv$ on the invariant two--sphere $Fix(\ZZ_2(\gamma_2)) \cap \EU^3$. Right: Open neighbourhoods of $\vv$ and $\ww$, inside which the direction of solutions turning around the connection $[\vv \rightarrow \ww]$ is the same. }
\label{orientation_exemplo}
\end{figure}

The network $\Sigma^0$ can be decomposed into two cycles. Due to the symmetry and to the asym\-ptotic stability, trajectories whose initial condition starts outside the invariant fixed point subspaces will approach in positive time one of the cycles. The fixed point hyperplanes prevent random visits to the two cycles; a trajectory that approaches one of the cycles in $\Sigma^0$ is shown in Figure \ref{simulation3(geom)}. The time series of the figure shows the increasing intervals of time spent near the equilibria. The sejourn time in any neighbourhood of one of saddle-foci increases geometrically with ratio $\rho$.

\subsection{Symmetry breaking along the article}\label{SymmetryAlong}
The symmetries of the organising centre are broken when either $\lambda_1$ or $\lambda_2$ is not zero. 
This was done by adding  two  terms $\lambda_1F_1(X)$ and $\lambda_2F_2(X)$ that break the $\textbf{SO(2)} \oplus \ZZ_2(\gamma_2)$ symmetry in different ways  --- see Appendix~\ref{AppendixTable}.
Both $F_1$ and $F_2$ are chosen to be  tangent to $\EU^3$, therefore $\XX(X,\lambda_1,\lambda_2)$ is still a well defined vector field on $\EU^3$.

In Section~\ref{Transverse}, after some additional results on $\XX_0(X)$, we  discuss briefly the dynamics of 
$\XX(X,0,\lambda_2)$, with $\lambda_2\ne 0$, when the rotational symmetry $\textbf{SO}(2)$ is broken, destroying the network $\Sigma^0$. 
Then we exploit the more dynamically interesting situation that arises for $\lambda_1\ne 0$ and $\lambda_2=0$.
In this case the reflection symmetry $\ZZ_2$ is broken, as well as part of the 
$\textbf{SO}(2)$--symmetry. 
The two-dimensional connection breaks into a pair of one-dimensional ones, 
but due to the remaining symmetry, the one-dimensional connections from $\vv$ to $\ww$ are preserved.
This gives rise to two Bykov cycles.
Analytical proof of transverse intersection of two invariant manifolds is usually difficult to obtain but this can be achieved in our example, and  is done in Appendix~\ref{AppendixTransversality}. Similar heteroclinic bifurcations have been studied in the context of a model of the long Josephson junction \cite{Berg}.

\begin{figure}
\begin{center}
\includegraphics[height=11cm]{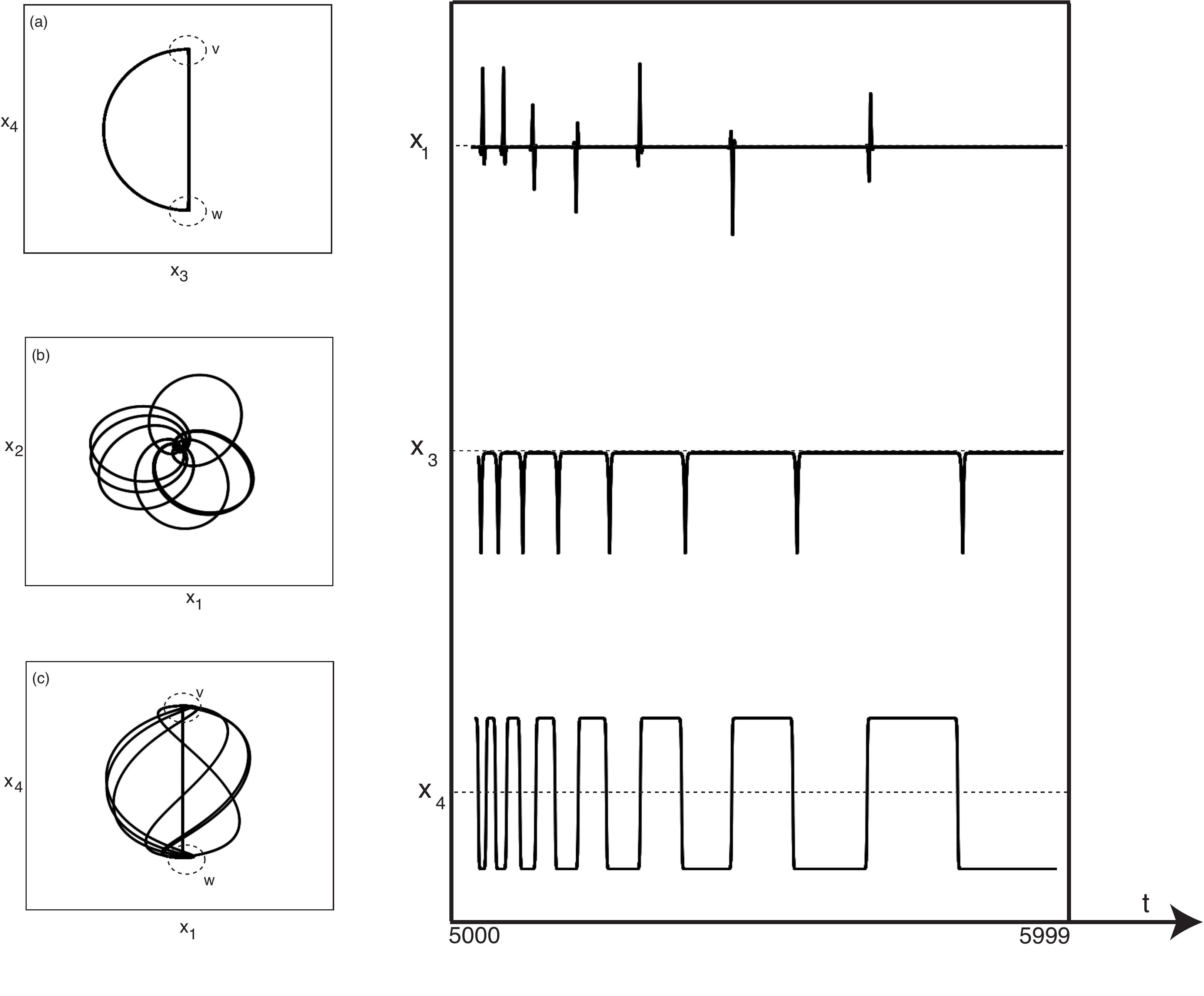}
\end{center}
\caption{\small  Left: Projection in the $(x_3,x_4)$, $(x_1,x_2)$ and $(x_1,x_4)$--planes of the trajectory with initial condition $(-0.5000, -0.1390, -0.8807, 0.3013)$ for the flow corresponding to $\XX_0(X)=\XX(X,0,0)$,
 with $\alpha_1=1$, $\alpha_2=-0.1$. Right: Corresponding time series.
The simulations omit the initial transient (the independent variable $t$ varies between 5000 and 5999). }
\label{simulation3(geom)}
\end{figure}

When both $\lambda_1$ and $\lambda_2$ are non zero, all the symmetry is broken,
hence the one-dimensional connection from $\vv$ to $\ww$ disappears.
Although the attracting heteroclinic cycles disappears, some nearby attracting structures remain, this is discussed in Section~\ref{BreakAllSymmetries}. Our results and conjectures are illustrated by numerical simulations, which have been obtained using {\emph{Matlab} and  the dynamical systems package \emph{Dstool}.}

\section{Partial symmetry breaking}\label{Transverse}
 The existence of the attracting heteroclinic network $\Sigma^0$ for $\XX_0$  requires two separate symmetries to allow structurally stable connections within the two invariant fixed point subspaces. Throughout this paper, we control the degree to which the two symmetries  
 are broken by two  parameters, $\lambda_1$ and $\lambda_2$: $\lambda_1$ controls the magnitude of the $\textbf{SO(2)} \times \ZZ_2 (\gamma_2)$--symmetry-breaking term and $\lambda_2$ controls the breaking of the pure reflectional symmetry  {$\ZZ_2 (\gamma_1)$ where $\gamma_1=\Psi_\pi \in \textbf{SO(2)}$}, $\gamma_1(x_1, x_2,x_3,x_4)=(-x_1, -x_2,x_3,x_4)$. The parameters $\lambda_1$ and $\lambda_2$ play exactly the same role as in \cite{LR}.
\medbreak
Before discussing the effects of breaking the symmetry in \eqref{example} when either 
$\lambda_1$ or $\lambda_2$ is non zero, we need some additional information on the fully symmetric case.

\subsection{The organising centre}\label{OrganisingCentre} 

For $\XX_0$, there are two possibilities for the geometry of the flow around each saddle-focus of the network $\Sigma^0$, depending on the direction solutions turn around $[\vv \rightarrow \ww]$, as discussed in Labouriau and Rodrigues \cite{LR} . The next proposition shows that each solution when close to $\vv$ turns in the same direction as when close to $\ww$. 

\begin{proposition}
\label{orientation}
In $\EU^3$, there are open neighbourhoods $V$ and $W$ of $\vv$ and $\ww$, respectively, such that, for any
trajectory  of $\XX_0$ going from $V$ to $W$, the direction of its turning around the connection $[\vv \rightarrow \ww]$ is the same in $V$ and in $W$.
\end{proposition}

\begin{proof}
The explicit expression of $\XX_0$ in spherical coordinates is the special case 
$\lambda_1=0$ of equation \eqref{system spherical coordinates} in Appendix~\ref{AppendixTransversality}.
The equation for  the angular coordinate $\varphi$ in the plane $(x_1,x_2,0,0)$ is
$\dot{\varphi}=1$.
Since this plane is perpendicular to $Fix(\textbf{SO(2)}(\gamma_1))$, where the connection  $[\vv \rightarrow \ww]$ lies, trajectories  must turn around the connection in the same direction.
\end{proof}

 The property in Proposition~\ref{orientation} is persistent under \emph{isotopies}: if it holds for the organising centre $\lambda_1=\lambda_2=0$, then it is still valid in smooth one-parameter families containing it, as long as there is still a connection.
 In particular, Property (P8) of \cite{LR} is verified and thus their results may be applied to the present work.

\subsection{Breaking the two-dimensional heteroclinic connection}
For $\lambda_1\ne 0$, the vector field $\XX_1(X)=\XX(X,\lambda_1,0)$ 
is no longer $\textbf{SO(2)}\oplus\ZZ_2(\gamma_2)$ but  is still equivariant under the action of $\gamma_1 \in \textbf{SO(2)}$.  This breaks the two-dimensional connection 
$[\ww\to \vv]$ into a transverse intersection of invariant manifolds.

\begin{theorem}
\label{TeoremaRede}
The vector field $\XX_1(X)=\XX(X,\lambda_1,0)$ on $\EU^3$ has symmetry group  
$\ZZ_2(\gamma_1)$ and 
for small $\lambda_1\ne 0$ {its flow}
has a heteroclinic network $\Sigma^\star$ involving the two equilibria $\vv$ and $\ww$ with the following properties:
\begin{enumerate}\item\label{OneDConnection}
there are two one-dimensional heteroclinic connections from $\vv$ to $\ww$ inside $Fix(\ZZ_2(\gamma_1)) \cap \EU^3$;
\item\label{homoclinic} there are no homoclinic connections to the equilibria;
\item
 the two-dimensional invariant manifolds of  $\vv$ and $\ww$ intersect  transversely  along one-dimensional connections from  $\ww$ to $\vv$.
\end{enumerate}
The connections  from $\vv$ to $\ww$ persist under small perturbations that preserve the $\gamma_1$--symmetry.
\end{theorem}

The new network $\Sigma^\star$ is qualitatively different from $\Sigma^0$. The { transversality}
of {the unstable manifold of $\ww$}, 
$W^u(\ww)$, and {the stable manifold of $\vv$}, $W^s(\vv)$, 
and the persistence of the heteroclinic connections from $\vv$ to $\ww$ give rise to a \emph{Bykov cycle}.
In {a} non-symmetric context, Bykov cycles arise as  bifurcations of codimension 2, the points in parameter space where they appear are called  \emph{T-points} in \cite{GS2} --- see Figure \ref{Bykov}.

 The existence of a Bykov cycle implies the existence of a bigger network: beyond the original transverse connections, there { are}
 infinitely many subsidiary heteroclinic connections turning around the original Bykov cycle,
 { that} 
 {have the same bifurcation structure in their unfolding. Multi-pulses are also expected.}

\begin{figure}
\begin{center}
\includegraphics[height=4cm]{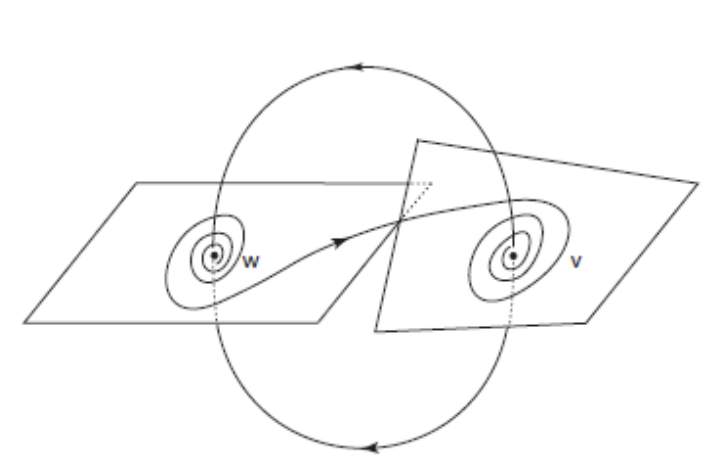}
\end{center}
\caption{\small Bykov cycle $\Sigma^\star$: heteroclinic cycle associated to two saddle-foci of different Morse indices, in which the one-dimensional invariant manifolds coincide and the two-dimensional invariant manifolds intersect  transversely .}
\label{Bykov}
\end{figure}

\begin{proof}[Proof of Theorem \ref{TeoremaRede}]
Clearly the perturbation term breaks all the symmetries of $\XX_1(X)$ except $\gamma_1$.
It is also immediate that $\gamma_1(\vv)=\vv$, $\gamma_1(\ww)=\ww$.
The circle $Fix(\ZZ_2(\gamma_1)) \cap \EU^3$ is still flow-invariant and contains no other equilibria, thus it still consists of the two equilibria  $\vv$ and $\ww$ and  the two connections from  $\vv$ to $\ww$. The plane $Fix(\ZZ_2(\gamma_1))$ will remain invariant under any perturbation that preserves the symmetry. Since on this plane $\vv$ is a saddle and $\ww$ is a sink, the connection persists under small perturbations. 

{Since $\vv$ and $\ww$ are hyperbolic, $Fix(\ZZ_2(\gamma_1))\supset[\vv \rightarrow \ww]$ and $\dim W^u(\vv)=\dim W^s(\ww) = 1$, item (\ref{homoclinic}) follows.} Breaking the $\ZZ_2(\gamma_2)$--equivariance is necessary for the existence of transverse intersection of the manifolds  $W^u(\ww)$ and $W^s(\vv)$. Nevertheless, the manifolds could intersect non  transversely .
The proof of transversality of the intersection of the two-dimensional manifolds, using the Melnikov method is similar to that given in Aguiar \emph{et al} \cite{ACL BIF CHAOS}. For completeness, we give the proof in Appendix~\ref{AppendixTransversality}.
\end{proof}

{Hereafter, our main goal is the description of the spiralling set that {
appears} in the flow of $\XX_1$. 
{First, we}
establish some results and terminology about the local and global dynamics. }
We follow {closely the general results in} \cite{ALR, LR, Rodrigues3}  {this is why proofs are short. More details are given in these references.}

\subsubsection{Local and global dynamics}\label{subLocalGlobal}
 By Theorem \ref{TeoremaRede}, the unfolding $\XX_1=\XX(X, \lambda_1, 0): \EU^3 \rightarrow \mathbf{T}\EU^3$ of $\XX_0$  is a family of $C^1$ vector fields such that for $\lambda_1 \neq 0$, the local two-dimensional manifolds $W^u(\ww)$ and $W^s(\vv)$ intersect  transversely  at (at least) two  trajectories. In order to describe the dynamics around the Bykov cycles, we introduce local coordinates near the saddle-foci $\vv$ and $\ww$ and we define some terminology.

Since $C_\vv \neq E_\vv$ and $C_\ww \neq E_\ww$, then by Samovol~\cite{Samovol}, the vector field $\XX_1$ is $C^1$--conjugate to its linear part around each saddle-focus. Linearisation may fail under resonance conditions that correspond to curves in the
 $(\alpha_1,\alpha_2)$--plane. This restriction has zero Lebesgue measure, and thus it does not place serious constraint on our analysis. In cylindrical coordinates $(\rho ,\theta ,z)$ the linear{isa}tion at $\vv$ takes the form:
$$
 \dot{\rho}=-C_{\vv }\rho \qquad  \dot{\theta}=1  \qquad  \dot{z}=E_{\vv }z
$$
and around $\ww$ it is given by:
$$
\dot{\rho}=E_{\ww }\rho \qquad \dot{\theta}= 1  \qquad \dot{z}=-C_{\ww }z .
$$
After a linear rescaling of the local variables, we  consider cylindrical neighbourhoods  $V\subset {\EU}^3$ and $W\subset {\EU}^3$  of $\vv $ and $\ww $, of radius $1$ and height $2$, respectively. We suggest that the reader follows this section observing Figure \ref{elipse}.

\begin{figure}[h]
\begin{center}
\includegraphics[height=7cm]{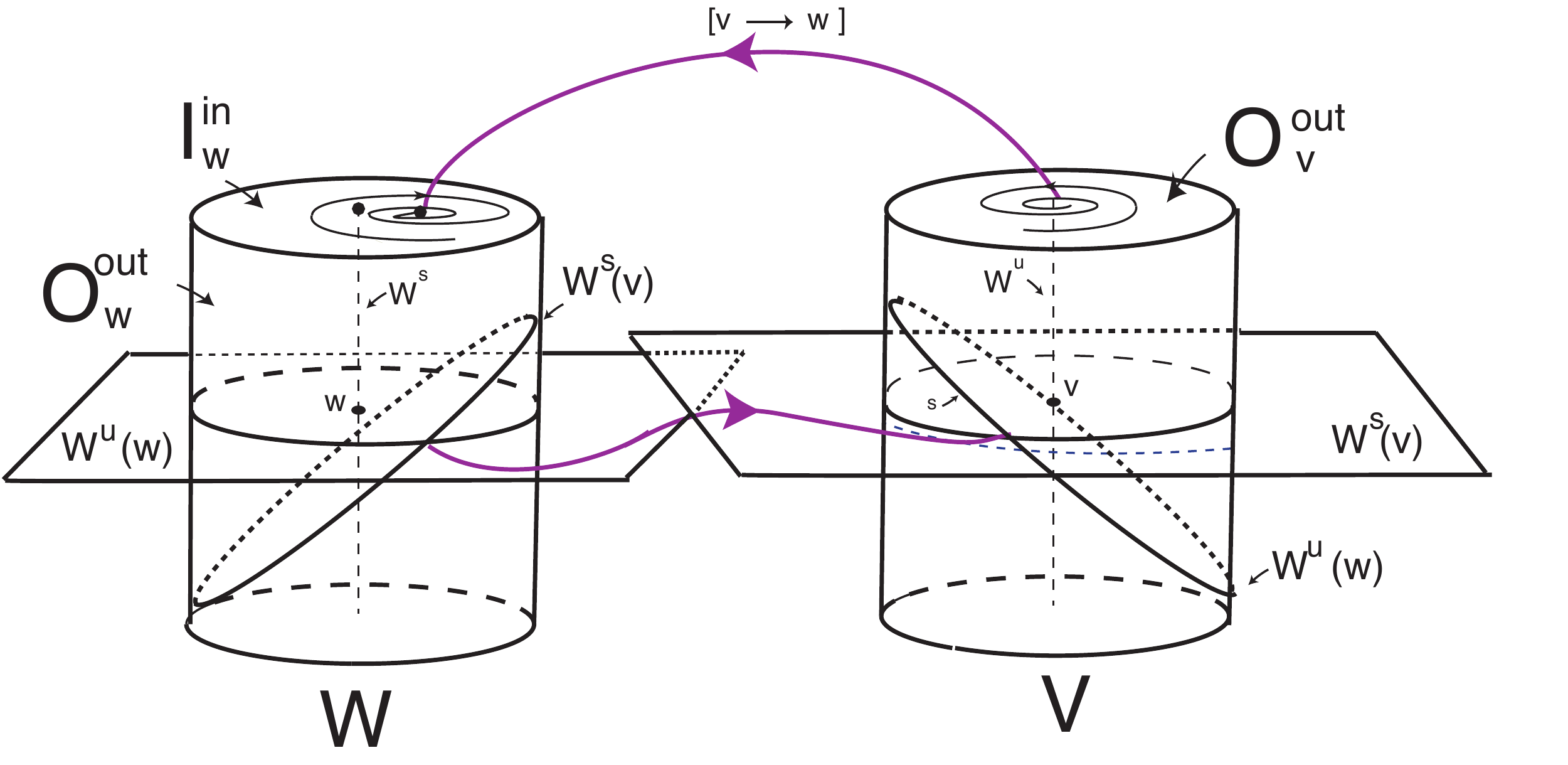}
\end{center}
\caption{\small 
Conventions in the cylindrical  neighbourhoods of $\vv$ and $\ww$. Both $W^s_{loc}(\vv) \cap O^{out}_\ww$ and $W^u_{loc}(\ww) \cap O^{in}_\vv$ are closed curves in  the boundaries of the cylinders for small values of $\lambda_1\ne 0$.}
\label{elipse}
\end{figure}

The boundary of each cylinder forms an \emph{isolating block}. 
 {It consists}
of three components: the cylinder wall parametr{ise}d by $x\in \RR\pmod{2\pi}$ and $|y|\leq 1$ with the usual cover $ (x,y)\mapsto (1 ,x,y)=(\rho ,\theta ,z)$ and two discs{, the top and bottom of the cylinder. 
Choosing coordinates so that one of the  connections $[\vv\rightarrow\ww]$  joins the tops of the two cylinders, we may from now on restrict our attention to the upper part of the walls $y\ge 0$ and the tops of the cylinders,  behaviour at the bottom is obtained by symmetry.
}
We take polar coverings of  
{the top discs $(r,\varphi )\mapsto (r,\varphi , 1)=(\rho ,\theta ,z)$ whith $0\leq r\leq 1$ with $\varphi \in \RR\pmod{2\pi}$
}
and let:
\begin{itemize}
\item $O_\vv^{out}$ be the  {top}
 of the cylinder $V$, where the flow goes out of $V$; 
\item $I_\vv^{in}$ be the  {upper part of the} wall of the cylinder $V$, where the flow goes in $V$;
\item $O_\ww^{out}$ be the {upper part of the}  wall of the cylinder $W$, where the flow goes out of  $W$;
\item $I_\ww^{in}$ be the {top} 
 of the cylinder $W$, where the flow goes in $W$.
\end{itemize}
The trajectories of  all points $(x,y)$ in $I_\vv^{in} \backslash W^s(\vv)$, leave $V$ at $O_\vv^{out}$ at
 {
\begin{equation}
\Phi_{\vv }(x,y)=\left(y^{\delta_\vv},-\frac{\ln y}{E_\vv}+x\right)=(r,\phi)
\qquad \mbox{where}\qquad 
\delta_\vv=\frac{C_{\vv }}{E_{\vv}}  \  .
\label{local_v}
\end{equation}
}
Similarly, points $(r,\phi)$ in $I_\ww^{in} \backslash W^s(\ww)$, leave $W$ at $O_\ww^{out}$ at
\begin{equation}
\Phi_{\ww }(r,\varphi )=\left(-\frac{\ln r}{E_\ww}+\varphi,r^{\delta_\ww}\right)=(x,y)
\qquad \mbox{where}\qquad 
\delta_\ww=\frac{C_{\ww }}{E_{\ww}} \ .
\label{local_w}
\end{equation}

The  transition between cylinders 
 {follows each connection  in a flow-box.}
The flow sends points in $O_\vv^{out}$ near $W^u_{loc}(\vv)$ into $I_\ww^{in}$ along the connection $[\vv\rightarrow\ww]$
 { defining a transition map $\Psi_{\vv \rightarrow \ww}:O_\vv^{out}\rightarrow I_\ww^{in}$
 that we may assume to be the identity.}
There is also a well defined transition map $ \Psi_{\ww \rightarrow \vv}:O_\ww^{out} \longrightarrow I_\vv^{in}$ that can be  {taken to be}
a rotation \cite{LR}.
 {We  choose our coordinates to have the connection $[\ww \rightarrow \vv]$ meeting $I_\vv^{in}$ and  $O_\ww^{out}$ at $(x,y)=(0,0)$.
 }

 { Let}
 $\eta:=  \Phi_{\ww }\circ \Psi_{\vv \rightarrow \ww} \circ \Phi_{\vv }$
  {and}
 define the first return map to $I_\vv^{in}$ as $ \Psi:=\Psi_{\ww \rightarrow \vv} \circ \eta$  at all points where it is well defined.

 {The geometry of the local transition maps is described using the following terminology:}
a \emph{ segment }$\beta $ 
on $I_\vv^{in}$ is a smooth regular parametr{ise}d curve 
$\beta :[0,1)\rightarrow I_\vv^{in}$ that meets $W^{s}_{loc}(\vv )$ transversely at the point $\beta (1)$ only and such that, writing $\beta (s)=(x(s),y(s))$, both $x$ and $y$ are monotonic functions of $s$.

A \emph{ spiral } on a disc $D$  around a point $p\in D$ is a curve 
$\alpha :[0,1)\rightarrow D$
satisfying $\dpt \lim_{s\to 1^-}\alpha (s)=p$ and such that if
$\alpha (s)=(r(s),\theta(s))$ is its expressions in
polar coordinates around $p$ then the maps $r$ and $\theta$ are monotonic, and 
$\lim_{s\to 1^-}|\theta(s)|=+\infty$.

Consider a cylinder $C$ parametr{ise}d by a covering $(\theta,h )\in  \RR\times[a,b]$,
with $a<b\in\RR$ where $\theta $ is periodic.
A \emph{helix} on the cylinder $C$ 
\emph{accumulating on the circle} 
$h=h_{0}$ is a curve
$\alpha :[0,1)\rightarrow C$
such that its coordinates $(\theta (s),h(s))$ 
satisfy $ \lim_{s\to 1^-}h(s)=h_{0}$, $\lim_{s\to 1^-}|\theta (s)|=+\infty$ and the maps $\theta$ and $h$ are monotonic.
Using these definitions and the expressions (\ref{local_v}) and (\ref{local_w}) for $\Phi_{\vv }$ and $\Phi_{\ww }$  we get:

\begin{proposition}[Aguiar \emph{et al} \cite{ALR}, 2010]
\label{Structures}
A segment on $I_\vv^{in}$ is mapped by $\Phi _{\vv}$ into 
 a spiral on $O_\ww^{out}$ around $W^u_{loc}(\vv)\cap Out(\vv) $.
 This spiral is mapped by $ \Psi_{\vv \rightarrow \ww}$  into another spiral around $W^s_{loc}(\ww)\cap I_\vv^{in}$,
 which is mapped by $\Phi _{\ww}$ into a helix on $O_\ww^{out}$ accumulating on the circle  $O_\ww^{out} \cap W^{u}(\ww)$.
\end{proposition}

\subsubsection{The spiralling set}
 {
 The study of ``routes to chaos'' has been a recurring concern in the research  on nonlinear dynamics during the last decades. 
 Several patterns have been described for  qualitative changes in features of vector fields that vary under small changes of a one-dimensional parameter. 
 In the present research, when $\lambda_1$ moves away from zero, we observe a 
 phenomenon called \emph{instant chaos} in different contexts \cite{Dawes, GWo, LD}: 
 an instantaneous jump from a  regular flow near $\Sigma^\star$ to chaotic behaviour, with suspended horseshoes and homoclinic classes.
 }

\begin{theorem}\label{TeoremaSimetrico}
For generic $\alpha_1 \neq \alpha_2$ and for any small $\lambda_1\ne 0$, the vector field $\XX_1=\XX(X,\lambda_1,0)$ is $\ZZ_2(\gamma_1)$--equivariant and {its flow} has a compact spiralling set $\Lambda\subset\EU^3$ containing the  {heteroclinic network $\Sigma^\star$ of Theorem~\ref{TeoremaRede},
involving the saddle-foci $\vv$ and $\ww$ with a transverse intersection of the two-dimensional invariant manifolds; and moreover the suspension of a compact set $\mathcal{H}=\bigcup_{i \in \mathbf{Z}} \mathcal{H}_i\subset I_\vv^{in}$, where $\{\mathcal{H}_i\}_{i \in \mathbf{Z}}$ is an increasing sequence of invariant sets, accumulating on $\Sigma^\star$.
The dynamics of the first return map $\Psi$ to $\mathcal{H}$ is uniformly hyperbolic and topologically conjugate to the suspension of a full shift over an infinite set of symbols.
}
\end{theorem}

\begin{proof}
 {The existence of the network $\Sigma^\star$ follows from Theorem \ref{TeoremaRede}.}
For $\lambda_1=0$, the heteroclinic network {$\Sigma^0$} for $\XX_0$ is asymptotically stable. 
In particular, there are arbitrarily small compact neighbourhoods of {$\Sigma^0$} such that the vector field is transverse to their boundaries where it points inwards. 
Let $\mathcal{N}$ be one of these neighbourhoods.  There exists $\lambda_\star>0$ such that for $\lambda_1 < \lambda_\star$, the vector field $\XX_1$ is still transverse to the boundary of $\mathcal{N}$ and $\Sigma^\star\subset \mathcal{N}$. 
Thus $\mathcal{N}$ is  a compact set that is positively invariant under the flow of $\XX_1$. Hence $\mathcal{N}$ contains an attractor.

We use the ideas of \cite{ACL NONLINEARITY, ALR, Rodrigues3} adapted to our purposes to show the existence of the suspended horseshoe  accumulating on the heteroclinic network (see Figure \ref{horseshoe}). We start by taking {cylindrical neighbourhoods $V$ and $W$ of each equilibrium as in \ref{subLocalGlobal}.}

\begin{figure}
\begin{center}
\includegraphics[width=10cm]{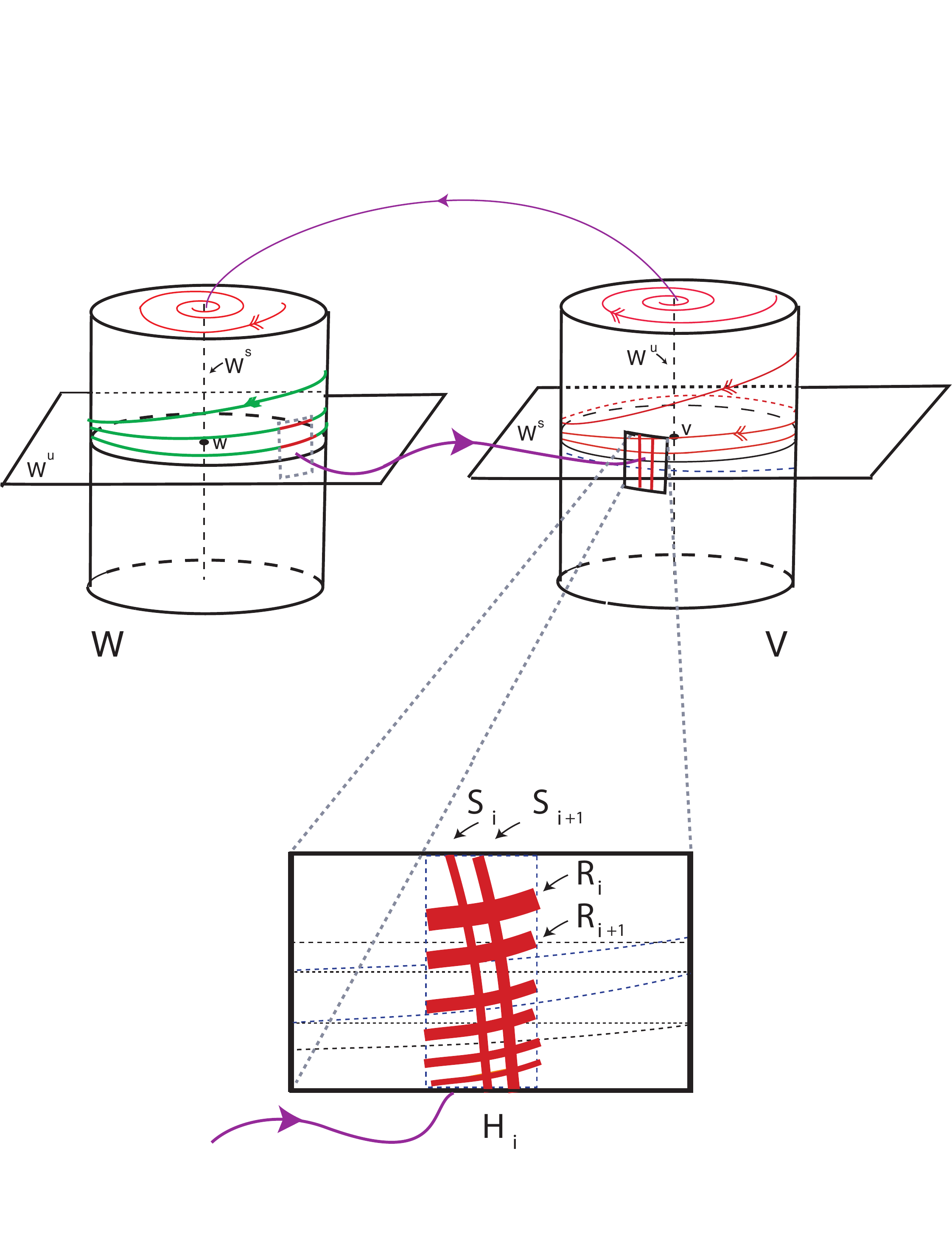}
\end{center}
\caption{\small Construction of the suspended horseshoe $\mathcal{H}_i$ for $\XX_1$ near the cycle, {giving rise} to the Cantor set on the wall of the cylinder, generated by an arbitrarily {large number of strips}. }
\label{horseshoe}
\end{figure}

{For small $\varepsilon >0$ and $\tau \in \left(0,1\right]$, consider the rectangles $R_\vv\subset I_\vv^{in}$ and $R_\ww\subset O_\ww^{out}$ parametrised by   $(x,y)\in[-\varepsilon, \varepsilon] \times [0,\tau]$.
}
Let $\beta$ be a vertical segment in $R_\vv$ whose angular component is  {$x_0 \in(-\varepsilon, \varepsilon)$,}
lying across the {stable manifold} of $\vv$. 
By Proposition \ref{Structures},  {its image by $\Phi _{\vv}$} accumulates as a spiral on the unstable manifold of $\vv$ 
 {that is then mapped by  $\Psi_{\vv \rightarrow \ww}$ into another spiral accumulating} on the stable manifold  {of} 
 $\ww$ --- see Figure \ref{horseshoe}. This spiral of initial conditions on {$I_\ww^{in}$}
 is mapped, by the local map near $\ww$, into points lying on a helix  {in $O_\ww^{out}$ that accumulates on $W^u_{loc}(\ww)\cap O_\ww^{out}$.
 On the other hand, the curve $W^s_{loc}(\vv)\cap O_\ww^{out}$ is a segment in $O_\ww^{out}$, hence the helix crosses it transversely infinitely many times.} 
%
The composition of consecutive local maps and transition functions transforms the original segment $\beta$ of initial conditions lying across  $W^s(\vv)$ into 
infinitely many segments in   $R_\vv$ with the same property. 

Proposition 5 of \cite{Rodrigues3} shows that there are $n_0 \in \textbf{N}$ and a family of intervals 
$(\mathcal{I}_n)_{n \geq n_0}= ([e^{a_n}, e^{b_n}])_{n \geq n_0} $, where:
$$
a_n= K^{-1} (-\varepsilon-2n\pi +x_0),  \quad \quad b_n=K^{-1} ( \varepsilon-2n\pi +x_0)\quad  \text{     and    }\quad  K= \frac{C_\vv+E_\ww}{E_\vv E_\ww} > 0,$$
such that
 {$\beta([e^{a_n},e^{b_n}])  \subset I_\vv^{in}$ and $\eta\circ\beta([e^{a_n}, e^{b_n}])\subset R_\ww$. }
Since $a_n$ and $b_n$ depend smoothly on the angular coordinate  {$x_0$} of the vertical segment $\beta$, then the sequence of horizontal strips:
$$
\mathcal{R}_n = [-\varepsilon, \varepsilon] \times  
 {\beta([e^{a_n(x_0)}, e^{b_n(x_0)}]) }
\subset R_\vv,  \quad n \geq n_0 \in \textbf{N}, \quad x \in[-\varepsilon, \varepsilon]
$$
 {is} mapped by $\eta$ onto $R_\ww$. The image under $\eta$ of the endpoints of the horizontal boundaries of $\mathcal{R}_n$ must join the end points of $\eta([e^{a_n}, e^{b_n}])\subset R_\ww$; since the map $\Psi_{\ww \rightarrow \vv}$ is assumed to be a rotation, for each $n > n_0$, the horizontal strip $\mathcal{R}_n$ is mapped by $\Psi$ into a vertical strip $\mathcal{S}_n$ across $R_\vv \subset I_\vv^{in}$.

 Moreover, the set 
$
\bigcap_{k \in \ZZ} \left(\bigcup_{i=1}^{k} \Psi^k (\mathcal{R}_i)\right)
$
is a Cantor set of initial conditions where the return map to $\bigcup_{i \in \textbf{N}} \mathcal{R}_i$ is well defined  {both} in forward and  {in} backward time, for arbitrarily large times. The dynamics of $\Psi$ restricted to the above invariant set is  {semi-conjugate}
to a full shift over an infinite alphabet that represents the paths in  {$\Sigma^\star$.}
 {
In \cite{ACL NONLINEARITY}, it is shown that that if $\Pi$ is a transverse section to the suspended horseshoe $\mathcal{H}$, the first return map to $\Pi$ is hyperbolic at all points where it is  well defined.
}
\end{proof}
%
%
%

Points lying on the invariant manifolds of $\vv$ and $\ww$ are dense in  { the suspended horseshoe} 
$\mathcal{H}$. 
 {In particular, the \emph{topological entropy} of the corresponding flow is positive.}
 {The set $\mathcal{H}$}
might have positive Lebesgue measure like the ``fat Bowen horseshoe''. {Using the Conley-Moser Conditions,} Rodrigues \cite{Rodrigues3} proved that the shift dynamics does not trap most trajectories  {that remain in a} 
small neighbourhood of $\Sigma^\star$:
\begin{corollary}
\label{zero measure}
Let ${N}^{\Sigma^\star}$  be a tubular neighbourhood  of one of the Bykov cycles $\Sigma^\star$.
Then, in any cross-section $\Pi_q\subset N^{\Sigma^\star}$ at a point  $q$ in $[\ww\rightarrow\vv]$, the set of initial conditions in $\Pi_q \cap {N}^{\Sigma^\star}$ that do not leave ${N}^{\Sigma^\star}$ for all time has zero Lebesgue measure.
\end{corollary}

\subsubsection{Heteroclinic Switching and Subsidiary Dynamics}
One astonishing property is the possibility of shadowing the heteroclinic network $\Sigma^\star$ by the property called \emph{heteroclinic switching}: any infinite sequence of pseudo-orbits defined by admissible heteroclinic connections  can be shadowed, as we proceed to define.

\begin{figure}
\begin{center}
\includegraphics[width=14cm]{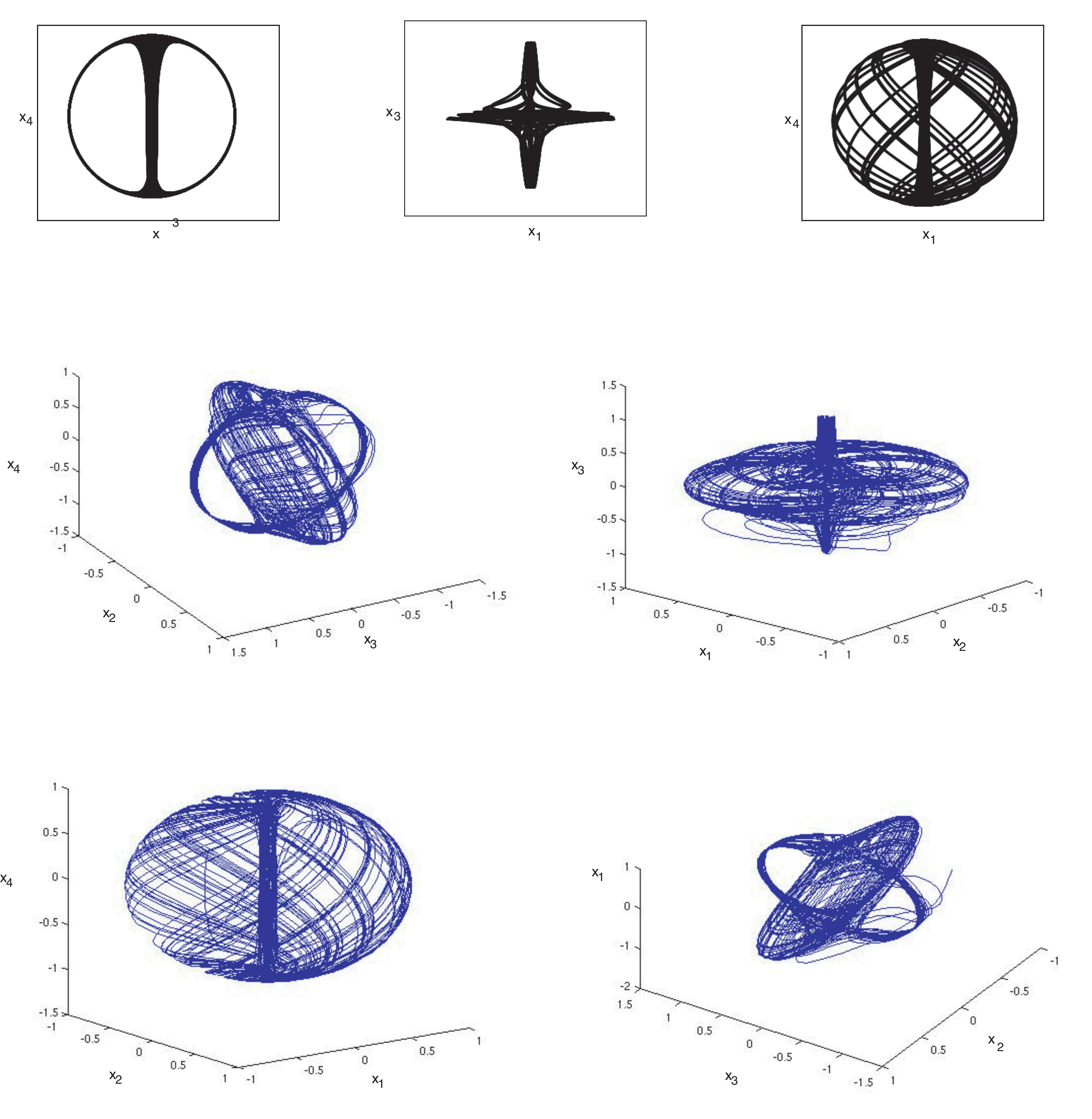}
\end{center}
\caption{\small The {structure of the } spiralling set $\Lambda$. Top: Projection in the $(x_3,x_4)$, $(x_1,x_3)$ and $(x_1,x_4)$--planes of the trajectory with initial condition $(-0.5000, -0.1390, -0.8807, 0.3013)$ for the flow of $\XX_1$, with $\alpha_1=1$, $\alpha_2=-0.1$ and $\lambda_1=0.05$. Centre and Bottom: Projection in the $(x_2, x_3,x_4)$, $(x_1,x_2, x_3)$, $(x_1,x_2, x_4)$ and $(x_1,x_3,x_4)$--hyperplanes of the trajectory with the same initial condition and parameters.}
\label{simulation1(geom)}
\end{figure}

\medbreak

 For the heteroclinic network $\Sigma^\star $ with node set $\{\vv, \ww\}$,
 a \emph{path of order k} \emph{\ }on $\Sigma^\star $ is a finite
sequence $s^k=(c_{j})_{j\in \{1,\ldots,k\}}$ of heteroclinic connections 
$c_{j}=[A_{j}\rightarrow B_{j}]$ in $\Sigma^\star $ such that 
 $A_{j}, B_{j} \in \{\vv, \ww\}$ and $B_{j}=A_{j+1}$. An infinite path corresponds to an infinite sequence of connections in $\Sigma^\star$. 
Let $N_{\Sigma^\star}$ be a neighbourhood of the network $\Sigma^\star$
and let $U_{A}$ be a neighbourhood of  {a}
node $A$. For each heteroclinic connection in $\Sigma^\star$, consider a point $p_i$ on it and a small neighbourhood $V_i$ of $p_i$. We assume that the neighbourhoods of the nodes are pairwise disjoint, as well for those of points in connections.
\medbreak
Given neighbourhoods as before, the trajectory $\varphi(t,q)$,  \emph{follows} the finite path
$s^k=(c_{j})_{j\in \{1,\ldots,k\}}$ of order $k$, if
there exist two monotonically increasing sequences of times 
$(t_{i})_{i\in \{1,\ldots,k+1\}}$ and $(z_{i})_{i\in \{1,\ldots,k\}}$
such that for all $i \in \{1,\ldots,k\}$, we have $t_{i}<z_{i}<t_{i+1}$ and:
\begin{enumerate}
\item
$\phi (t,q)\subset N_{\Sigma^\star}$ for all $t\in (t_{1},t_{k+1});$
\item
$\phi (t_{i},q) \in U_{A_{i}}$ and $\phi (z_{i},q)\in V_{i}$ and
\item
 for all $t\in (z_{i},z_{i+1})$, $\phi (t,q)$ does not visit the neighbourhood  of any other node except that of $A_{i+1}$.
\end{enumerate}
There is \emph{finite switching} near $\Sigma^\star$ if  for each finite path there is a trajectory that follows it. Analogously, we define \emph{infinite switching} near $\Sigma^\star$ if every forward infinite sequence of connections in the network is shadowed by nearby trajectories. 

Infinite switching near $\Sigma^\star$ follows from the proof of Theorem \ref{TeoremaSimetrico} and the results of Aguiar \emph{et al} \cite{ALR}. The solutions that realise switching lie in a tubular neighbourhood ${N}^{\Sigma^\star}$ of $\Sigma^\star$, hence from the results of Rodrigues \cite{Rodrigues3}, it follows that:

\begin{figure}
\begin{center}
\includegraphics[height=8cm]{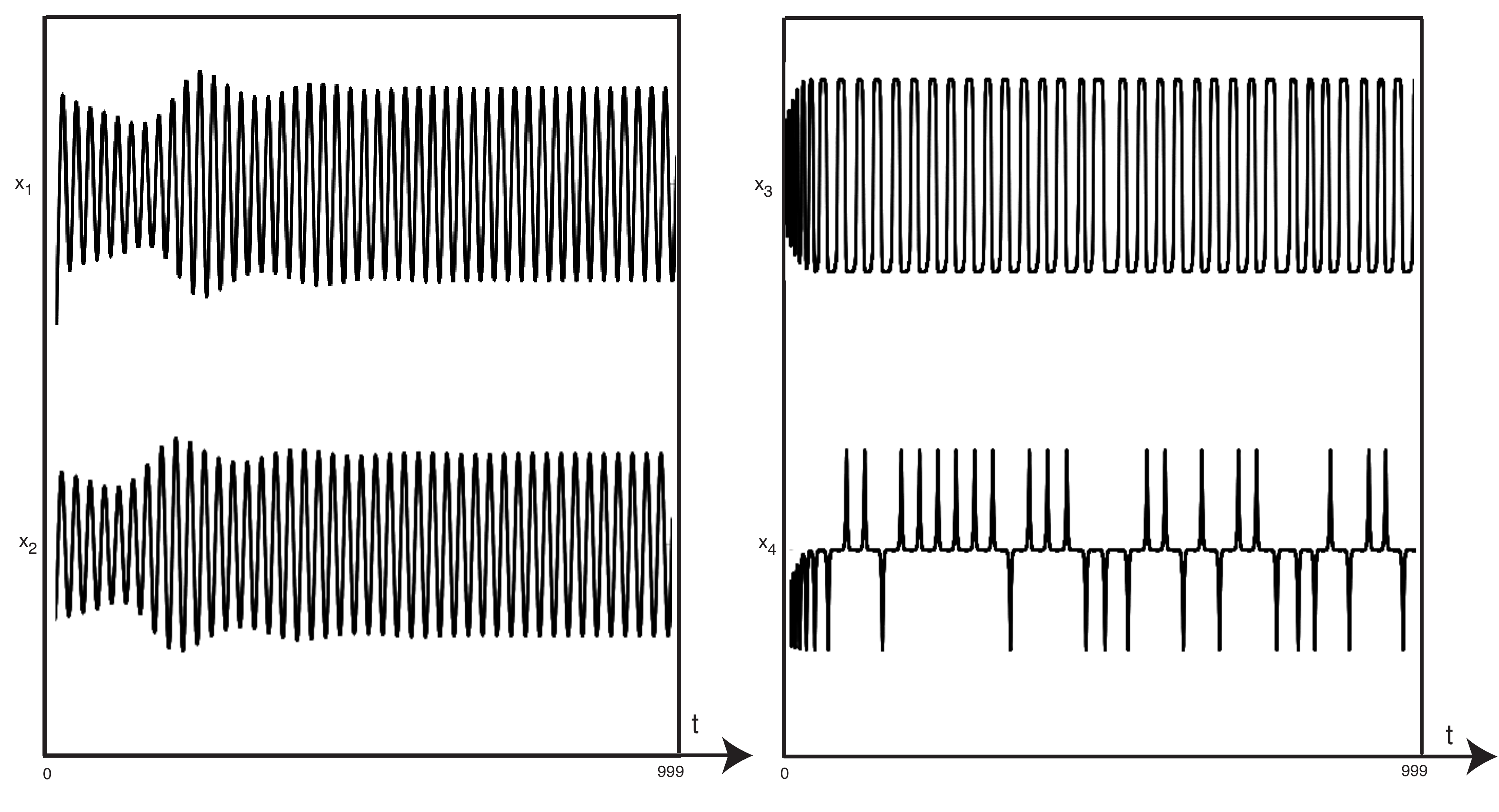}
\end{center}
\caption{\small Time series for the trajectory with initial condition $(-0.5000, -0.1390, -0.8807, 0.3013)$ for the flow of $\XX_1$, with  $\alpha_1=1$, $\alpha_2=-0.1$ and $\lambda_1=0.05$. }
\label{simulation1(ts)}
\end{figure}

\begin{corollary}
\label{Switching}
 There is  a set of initial conditions with positive Lebesgue measure for which there is finite switching near the network $\Sigma^\star$. 
 There is also infinite switching, real{ise}d by a set of initial conditions with zero Lebesgue measure.
\end{corollary}

The complex eigenvalues force the spreading of solutions around the unstable manifold of $\ww$, allowing visits to all possible connections starting at $\ww$. The { transversality} enables the existence of solutions that follow heteroclinic connections on the two different connected components of $\textbf{S}^3 \backslash W^s_{loc}(\vv)$, the upper and the lower part on the wall of the cylinder --- see Figure \ref{horseshoe}.

From  the results of \cite{LR}, it also follows that  near $\Sigma^\star$ the only heteroclinic connections from $\vv $ to $\ww $ are the original ones and that the finite switching near $\Sigma^\star$ may be realised by an  $n$--pulse heteroclinic connection $[\ww\to\vv]$.
Moreover, among all the solutions which appear in $\mathcal{H}$, there are infinitely many knot types, inducing all link types.

\begin{figure}
\begin{center}
\includegraphics[height=12cm]{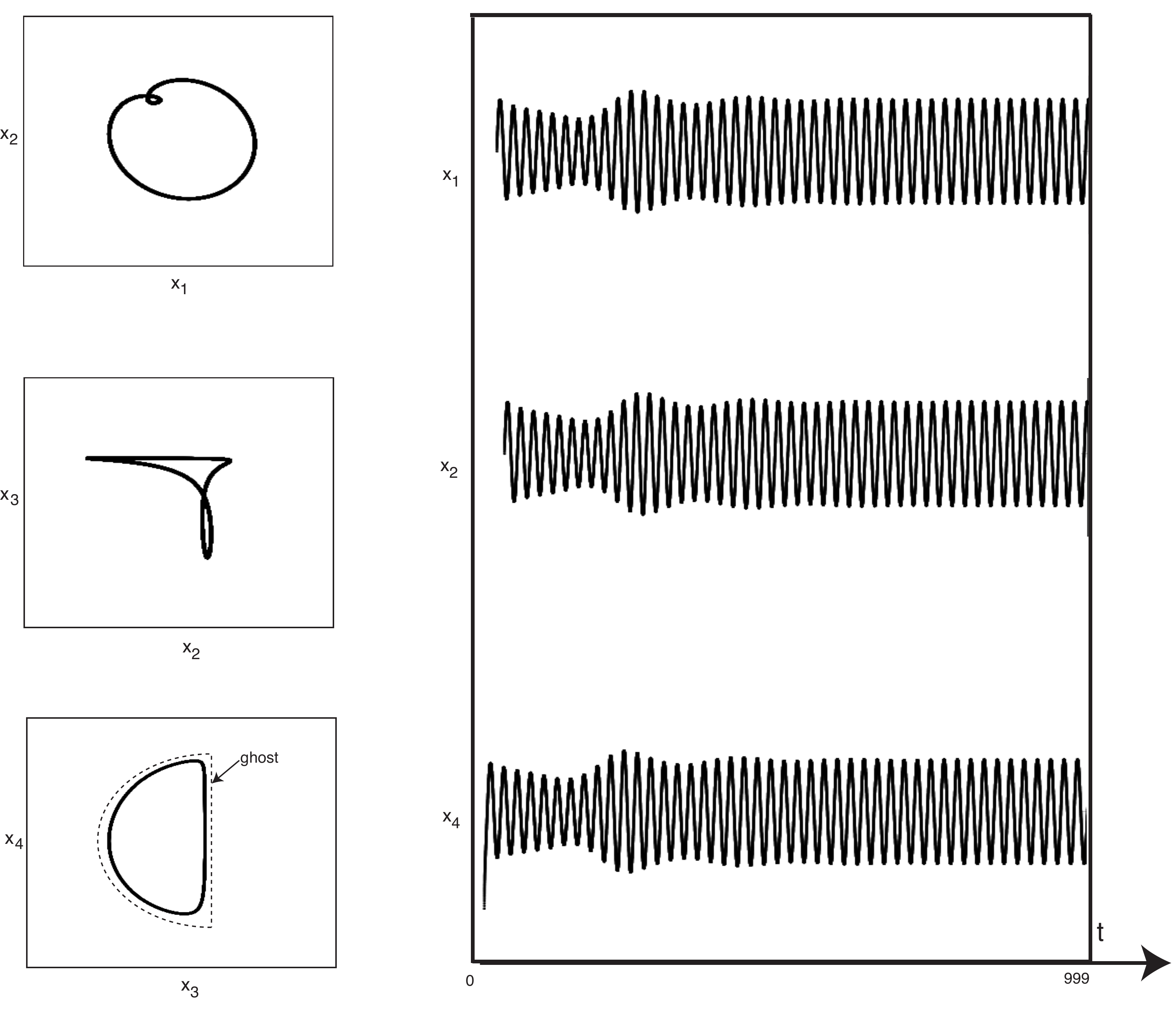}
\end{center}
\caption{\small  {Example of a trajectory that accumulates on an attracting periodic orbit.}
Left: Projection in the $(x_1,x_2)$, $(x_2,x_3)$ and $(x_3,x_4)$--planes of the trajectory with initial condition $(-0.5000, -0.1390, -0.8807, 0.3013)$ for the flow corresponding to the equation \eqref{example}, with $\lambda_1=0$, $\lambda_2=0.05$ for  $\alpha_1=1$ and  $\alpha_2=-0.1$. The dotted line on the $(x_3,x_4)$--plane indicates the position of the original cycle. Right: Time series for the corresponding trajectory.}
\label{simulation2(geom)}
\end{figure}

We illustrate the chaotic behaviour
in the {projected} phase portraits of Figure \ref{simulation1(geom)} and in the time series of
Figure~\ref{simulation1(ts)} corresponding to a trajectory that stays near the heteroclinic network. 
Observing Figure \ref{simulation1(geom)}, it is clear why we call the nonwandering set 
$\Lambda$ of $\XX_1$ a {\em spiralling set}.
The time series of Figure~\ref{simulation1(ts)} also  {suggests} heteroclinic switching: the trajectory follows a sequence of heteroclinic connections in a random order. 
We stress that {this 
occurs because the $\ZZ_2(\gamma_2)$--symmetry has been broken.}

Under generic perturbations (not necessarily equivariant), there is still an invariant topological sphere, since $\textbf{S}^3$ is normally hyperbolic, and the transverse connections are preserved. The spiralling set presented in Theorem~\ref{TeoremaSimetrico} may lose branches, changing its nature. 
 {Any compact, positively-invariant neighbourhood  of the original heteroclinic network will still be positively-invariant after perturbation and will then contain  an uniformly hyperbolic basic set displaying structurally stable homoclinic classes in its unfolding, where the results of Rodrigues \emph{et al} \cite{RLA} may be applied.
}
We will return to this issue in Section~\ref{BreakAllSymmetries} below.

\subsubsection{Nonhyperbolic behaviour}
Based in \cite{LR_proc}, applying the construction of the proof of Theorem \ref{TeoremaSimetrico} to the unstable manifold of $\ww$, we obtain the following result:

\begin{theorem}
\label{tangencies}
There are values  $\lambda_1^\star$  {of $\lambda_1$} arbitrarily close to 0, for which the flow of $\XX(X, \lambda_1^\star, 0)$ has a heteroclinic tangency between $W^u(\ww)$ and $W^s(\vv)$, coexisting with the transverse connections in $\Sigma^\star$. 
\end{theorem} 
 
The tangencies of Theorem \ref{tangencies} take place outside
the invariant set $\mathcal{H}$ of Theorem \ref{TeoremaSimetrico}. 
Since $E_\vv<C_\vv$, Newhouse's results \cite{Newhouse74, Newhouse79}   on homoclinic tangencies, extended to heteroclinic tangencies by Hayashi's
Connecting Lemma  \cite{Hayashi}, ensure the existence of infinitely many sinks nearby.

\begin{proof}
 {Let} $V$ and $W$ be  {the} cylindrical neighbourhoods of each equilibrium point 
defined in \ref{subLocalGlobal}.
 { The unstable manifold $W^u_{loc}(\ww)$ intersects the whole cylinder wall $I_\vv^{in}$ (not only the upper part)
 on a closed curve,}
as shown in Figure \ref{elipse}. For small $\lambda_1 \neq 0$, {the portion of the curve lying in the component of $I_\vv^{in}$ with $y>0$} has a point of maximum height that divides it in two components. {By Proposition \ref{Structures}, each one of these components} is mapped into a helix around $O_\ww^{out}$. The two helices taken together form a curve with  {at least one}
fold point (see Figure \ref{fold}). Varying $\lambda_1$ moves the fold point around $O_\ww^{out}$, exponentially fast in $\lambda_1$. As $\lambda_1$ tends to zero, the fold point will cross $W^s_{loc}(\vv) \cap O_\ww^{out}$ infinitely many times, creating heteroclinic tangencies.
\end{proof}

Each heteroclinic tangency  may be
 destroyed  locally by a small perturbation.
 As reported in \cite{LR_proc}, these tangencies correspond to the intersection of the local stable and unstable manifolds of a modified horseshoe
with infinitely many slabs, and they may be seen  as the product of Cantor sets with the property that the fractal dimension is large. This thickness is essentially due to the existence of infinitely many attracting solutions. 
 
{These} tangencies coexist with the hyperbolic  {Cantor} set $\mathcal{H}$ and the basins of  {attraction}
of the sinks 
 {lie} in the gaps of $\mathcal{H}$.

\begin{figure}
\begin{center}
\includegraphics[height=4cm]{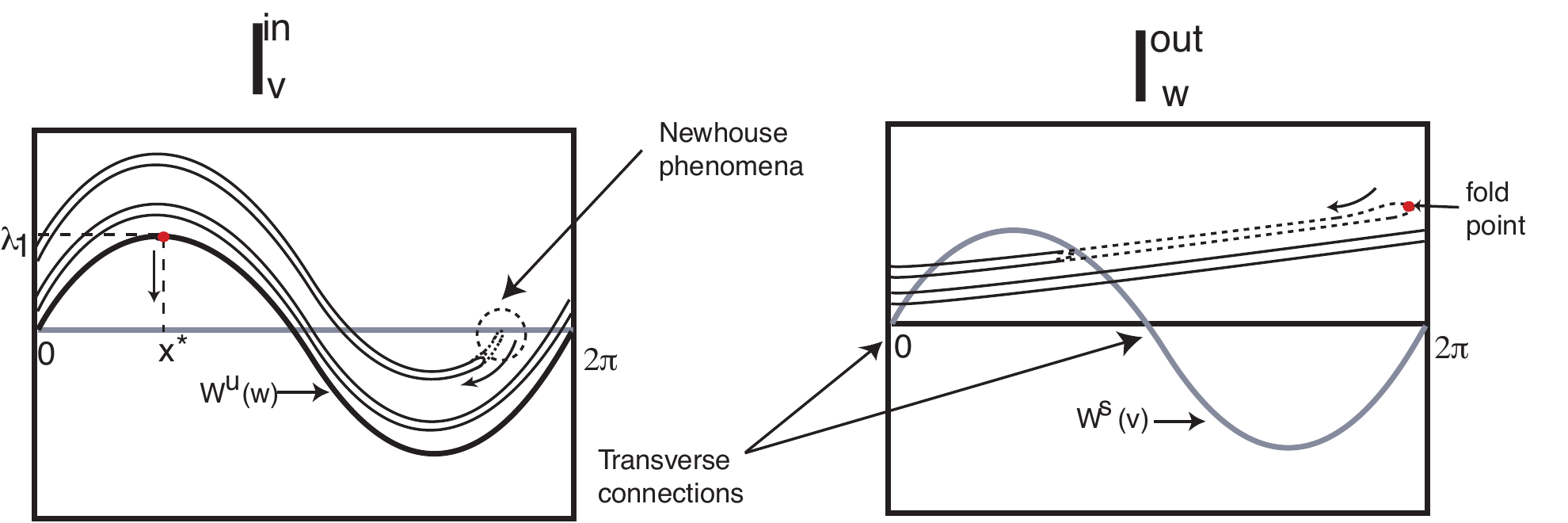}
\end{center}
\caption{\small Heteroclinic tangency between $W^u(\ww)$ and $W^s(\vv)$, coexisting with the transverse connections in $\Sigma^\star$.}
\label{fold}
\end{figure}

\subsection{Breaking the one-dimensional heteroclinic connection}

For $\lambda_2\ne 0$, the vector field $\XX_2=\XX(X,0,\lambda_2)$ 
is no longer $\textbf{SO(2)}$--equivariant  but  is still equivariant under the action of $\gamma_2$. 
Since the heteroclinic connections $[\vv \rightarrow \ww]$ lie on 
$Fix(\ZZ_2(\gamma_1))$,
these connections disappear, destroying the heteroclinic network $\Sigma^0$, but for  small values of $\lambda_2$
there will still be an attractor lying close to the original cycle.  

Theorem 6 of \cite{LR} shows that for sufficienly small $\lambda_2 \neq 0$, each heteroclinic cycle that occurred in the fully symmetric case is replaced by a stable
hyperbolic periodic solution. Using the reflection symmetry $\gamma_2$, two stable periodic solutions co-exist, one in each connected component of $\EU^3 \backslash Fix(\ZZ_2(\gamma_2))$. Their period tends to $+\infty$ when $\lambda_2 \rightarrow 0$ and their basin of attraction must contain the basin of attraction of $\Sigma^0$. For sufficiently large $\lambda_2$, saddle-node bifurcations of these two periodic solutions occur. Figure \ref{simulation2(geom)} illustrates the existence of a single attracting periodic solution in each connected component $\EU^3 \backslash Fix(\ZZ_2(\gamma_2))$ of the phase space.

\section{Breaking all the symmetry}\label{BreakAllSymmetries}
\subsection{Linked homoclinic cycles}
In this section, we describe the global bifurcations that are most important for our analysis, which appear when both $\lambda_1$ and $\lambda_2$ are non-zero and $\XX$ has no symmetry. Different chaotic dynamics in 
\eqref{example}  organise a complex network of bifurcations involving {the} rotating nodes, that had not yet been considered. The next theorem describes  the bifurcation diagram for $\XX$, shown in  Figure~\ref{BifDiagram}.

\begin{theorem}\label{ThHorseshoes}
The germ at the origin of the bifurcation diagram on the $(\lambda_1, \lambda_2)$--plane for $\XX(X,\lambda_1 \lambda_2)$ when $\lambda_1\lambda_2\ne 0$ satisfies:
\begin{enumerate}
\item\label{FiniteShoes}
For each $(\lambda_1, \lambda_2)$, there is 
a suspended uniformly hyperbolic horseshoe, topologically conjugate to a full shift over a finite number of symbols.
As $(\lambda_1, \lambda_2)$ tends to $(\lambda_1,0)$ with $\lambda_1 \ne 0$, the horseshoe accumulates on
the set $\mathcal{H}$ of Theorem~\ref{TeoremaSimetrico}.
\item\label{vHomo}
{There exists} a countable family of open tongues, tangent at $(0,0)$ to the $\lambda_2$--axis, where $\XX$ has an attracting periodic orbit. The closures of these tongues are pairwise disjoint. 
At the two curves comprising the boundary of each tongue there is an attracting Shilnikov homoclinic connection at $\vv$.
\item\label{wHomo}
{There exists} a countable family of open tongues, tangent at $(0,0)$ to the $\lambda_1$--axis, where $\XX$ exhibits a suspended  horseshoe, topologically conjugate to a full shift over an finite set of symbols. The closures of these tongues are pairwise disjoint. 
At the two curves comprising the boundary of each tongue there is an unstable Shilnikov homoclinic connection at $\ww$.
\item\label{twoHomo}
The two sets of curves in the parameter plane described in \eqref{vHomo} and \eqref{wHomo} meet at isolated points where the two types of homoclinic connections coexist. These codimension two points accumulate at the origin.
The closure of the pair of homoclinic orbits forms a link whose linking number tends to infinity as the points approach the origin of  the $(\lambda_1, \lambda_2)$--plane.
\end{enumerate}
\end{theorem}

\begin{figure}
\begin{center}
\includegraphics[width=135mm]{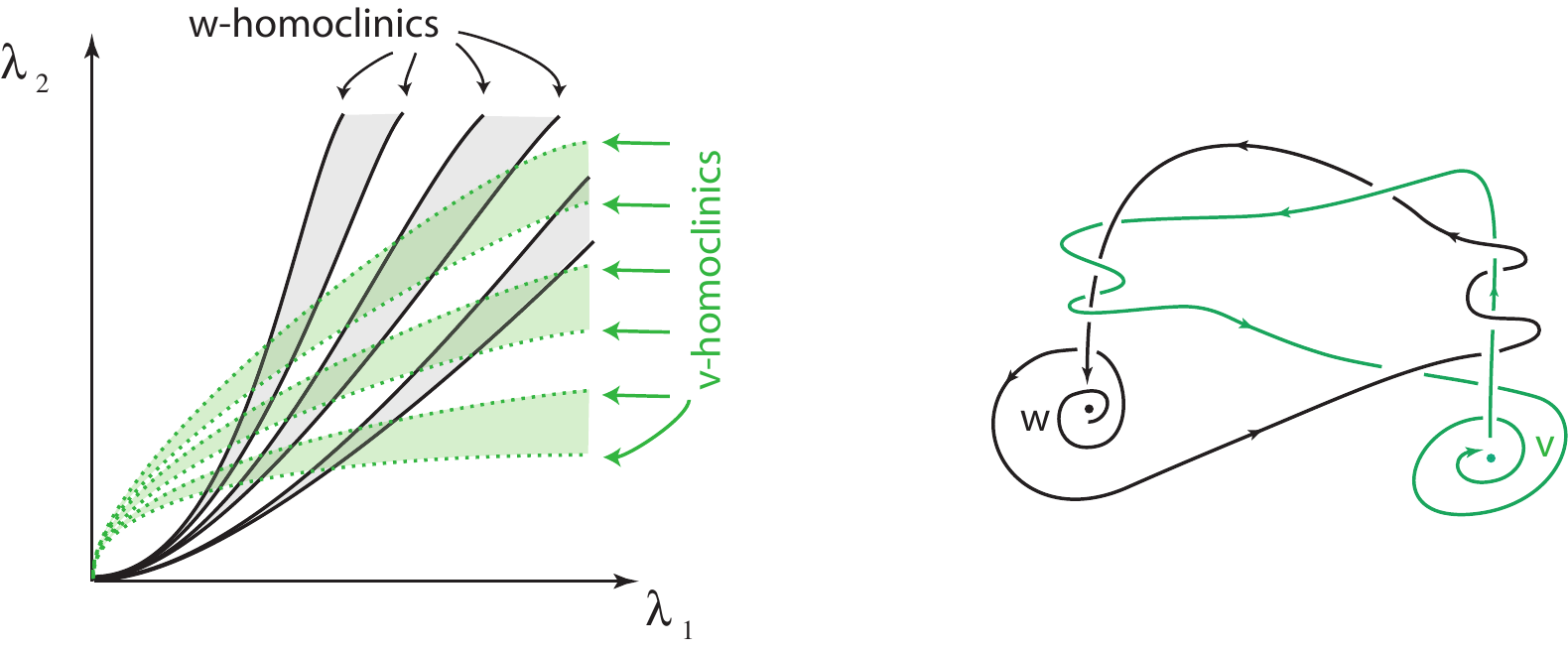}
\end{center}
\caption{\small Left: Qualitative bifurcation diagram for the family $\XX(X,\lambda_1, \lambda_2)$. 
At the solid curves tangent to the $\lambda_1$ axis there is an expanding Shilnikov  homoclinic at $\ww$. Consecutive pairs of these curves bound a tongue (shaded) where $\XX$ has infinitely many unstable periodic solutions associated to a suspended  horseshoe.
Contracting Shilnikov  homoclinics at $\vv$ appear at the dotted curves tangent to the $\lambda_2$ axis, that limit tongues where $\XX$ has an attracting period orbit. 
Right: Two linked Shilnikov  homoclinic loops appear when the two types of curve meet.}
\label{BifDiagram}
\end{figure}

Theorem~\ref{ThHorseshoes} follows  from Theorems 7 and 8 in \cite{LR}. 
We  {present}
an alternative proof, emphasising the transition between the two different types of {chaotic regimes and their geometry.
The  proof }can be carried out for Bykov cycles that are not symmetric. Symmetry is not a deciding factor in the creation of horseshoes, but it makes their existence more natural.

\begin{proof}
In  Theorem~\ref{TeoremaSimetrico}, for $\lambda_2=0$, we have found a  sequence $\mathcal{H}$ of suspended horseshoes  contained in the spiralling set $\Lambda$.
When  generic perturbation terms are added, the hyperbolic horseshoes $\mathcal{H}$ lose infinitely many branches,  only a finite number of  strips survive, as {stated} in  {\em \eqref{FiniteShoes}}. 

When $\lambda_1 \neq 0$ and $\lambda_2=0$,  {we have shown in Proposition~\ref{Structures} that} a segment $\beta$ of initial conditions in $I_\vv^{in}$, with $\beta$ transverse to $W^s(\vv)$, is mapped into a spiral in 
$I_\ww^{in}$ and then into a helix in  $O_\ww^{out}$ accumulating on $W^u(\ww)$.

 When  $\lambda_2 \neq 0$,  { the one-dimensional connection in the cycle is broken. The} 
spirals in $I_\ww^{in}$ are off-centred  (case (a) of Figure \ref{perturbation1_horseshoes}) and will turn
 {only} a finite number of times around $W^s(\ww)$. This is why the suspended horseshoes  in assertion {\em \eqref{FiniteShoes}} have a finite number of strips.
Suspended horseshoes arise when $W^u(\vv)$ is close to 
$W^s(\ww) \cap  I_\ww^{in}$, the number of  {strips}
increases  {as} 
$W^u(\vv)$ approaches  {$W^s(\ww)$.}

From the coincidence of the invariant two-dimensional manifolds near the equilibria in the fully symmetric case, we expect that when $\lambda_1$ is close to zero, $W^s(\vv)$ intersects the wall $O_\ww^{out}$ of the cylinder $W$ in a closed curve as in Figure \ref{elipse}. 
When both symmetries are broken, the Bykov cycle $\Sigma^\star$ is destroyed, giving rise to Shilnikov homoclinic cycles involving the saddle-foci $\vv$ and $\ww$. The equilibria have different Morse indices, thus the dynamics near each homoclinic cycle is qualitatively different, see 
Shilnikov \cite{Shilnikov65, Shilnikov68}:
 all the homoclinic orbits associated to $\vv$ are attracting since the contracting eigenvalue is larger than the expanding one. Each homoclinic cycle associated to $\ww$ has a suspended horseshoe near it and thus infinitely many periodic solutions of saddle type  occur in every neighbourhood of the homoclinicity of $\ww$.

The existence of homoclinic orbits is not a robust property, they occur along the curves in the  $(\lambda_1, \lambda_2)$--plane described in assertions \emph{ \eqref{vHomo}} and \emph{\eqref{wHomo}} --- details on these curves are given in \cite[Section 6]{LR}. 
 {Explicit approximate expressions for these curves may be obtained from the linear part of the vector field, using the transition maps obtained in \ref{subLocalGlobal}.}
At values of $(\lambda_1, \lambda_2)$ where the two types of curves cross, 
as in assertion \emph{\eqref{twoHomo}},  the  homoclinics at $\vv$ and $\ww$ occur at different {regions} in phase space.

\begin{figure}[h]
\begin{center}
\includegraphics[width=16cm]{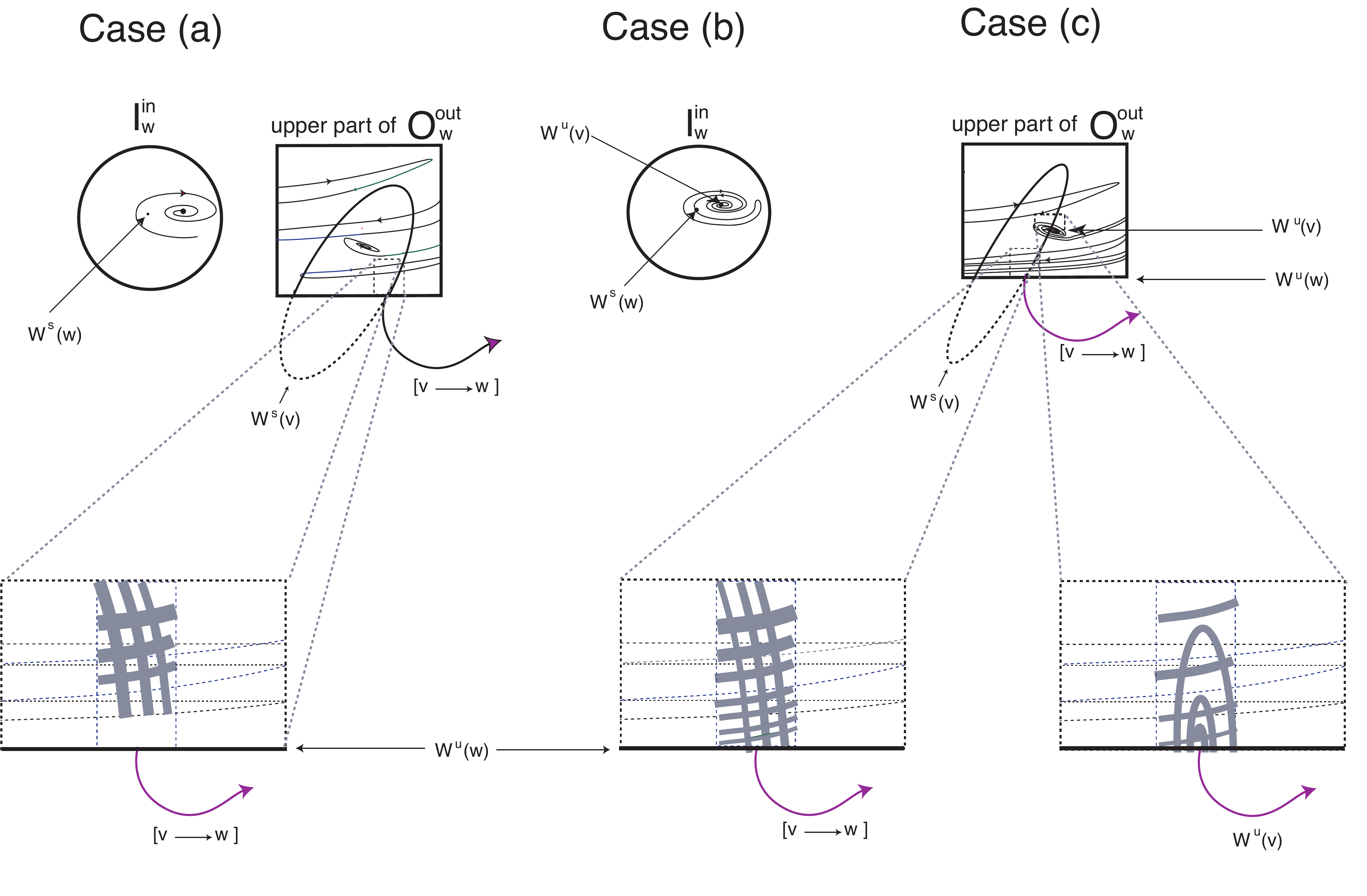}
\end{center}
\caption{\small There is a sequence of values of $\lambda_2$ for which horseshoes with infinitely many \emph{legs} appear and disappear as in a dance, being part of a very complex spiralling network of sets near the ghost of $\Sigma^\star$.
(a) A segment in $I_\vv^{in}$ transverse to $W^s(\vv)$ is mapped into a spiral in $I_\ww^{in}$. When $\lambda_2$ varies, branches of this curve get closer to $W^s(\ww)$ and its image makes more turns around $O_\ww^{out}$, creating horseshoes with more strips.
(b) When one of the arms of the  double spiral $W^u(\ww)\cap I_\ww^{in}$ touches $W^s(\ww)$ a homoclinic loop at $\ww$ is created.
(c) The double spiral  $W^u(\ww)\cap I_\ww^{in}$ is mapped into a curve in $O_\ww^{out}$ that winds a few times around the cylinder and then spirals into a point in $W^u(\vv)$. A stable homoclinic loop at $\vv$ is created when the spiral centre touches the closed curve 
$W^s(\vv)\cap O_\ww^{out}$. This may happen for the same parameter values that yield the homoclinic at $\ww$.   }
\label{perturbation1_horseshoes}
\end{figure}

Case (b) of Figure~\ref{perturbation1_horseshoes} describes the homoclinic cycle at $\ww$: 
initially $W^u(\ww)$ meets $I_\vv^{in}$ at a closed curve, that is then mapped into a double spiral in $I_\ww^{in}$ centred at $W^u(\vv)\cap I_\ww^{in}$.
When this spiral meets $W^s(\ww) \cap  I_\ww^{in}$, an unstable homoclinic cycle associated to $\ww$ is created.
Similarly, the backwards image of $W^u(\vv)$ in $O_\ww^{out}$ is a closed curve, 
that iterates back into a double spiral in $I_\ww^{in}$ centred at $W^s(\ww)\cap I_\ww^{in}$,
creating  an asymptotically stable homoclinic orbit of $\vv$ when this spiral meets 
$W^u(\vv)\cap I_\ww^{in}$. The values of $(\lambda_1, \lambda_2)$ where these intersections happen correspond to the two types of curves in Figure~\ref{BifDiagram}.
The double spiral $W^u(\ww)\cap I_\ww^{in}$ is mapped into a spiral in $O_\ww^{out}$ as in Figure~\ref{perturbation1_horseshoes} (c). 
As one of the spiral arms gets closer to $W^s(\ww)$ in $I_\ww^{in}$, its image winds around 
$O_\ww^{out}$.

Along the curve  {in the parameter plane} where one of the arms of the double spiral  in  $W^u(\ww)\cap I_\ww^{in}$ touches  {$W^s(\ww)\cap I_\ww^{in}$,}
there are some parameter values where the centre of the spiral in $O_\ww^{out}$ meets the closed curve $W^s(\vv)\cap O_\ww^{out}$ (Figure~\ref{perturbation1_horseshoes} (c)).
At these points both homoclinic cycles coexist in different regions of the phase space and
the basin  of attraction of the homoclinic  at $\vv$ lies inside the gaps of the  suspended horseshoes that appear and disappear from the homoclinicity at $\ww$. 
\end{proof}

We observe a mixing of regular and chaotic dynamics, in  regions of parameter space $(\lambda_1, \lambda_2)$ where the two types of tongues overlap.  Examples of chaotic behaviour of this type should be regarded as quasiattractors due to the presence of stable periodic orbits within them. A lot more needs to be done before these transitions are well understood.

\subsection{Finite heteroclinic switching}
Generic breaking of the $\gamma_1$--symmetry destroys the heteroclinic network, because the two connections $[\vv\rightarrow\ww]$ break. Nevetheless, finite switching might be observed: for small values of $\lambda_2$ there are trajectories that visit neighbourhoods of finite sequences of nodes. This is because the spirals on top of the cylinder around $\ww$ are off-centred and will turn a finite number of times around $W^s(\ww)$. In this way we may obtain points whose trajectories follow short finite paths on the network. As $W^u(\textbf{v})$ gets closer to $W^s(\textbf{\ww})$ (as the system  moves closer to $\ZZ_2(\gamma_1)$ symmetry) the paths that can be shadowed get longer.

\section{Discussion}


The goal of this paper is to construct explicitly a two-parameter  {family of} polynomial differential equation{s}, in which each parameter controls a type of symmetry breaking. We prove analytically the transverse intersection of two invariant manifolds using the adapted Melnikov method which is very difficult in general. 

 {This article uses a systematic method to construct examples of vector fields with simple forms that make their dynamic properties amenable to analytic proof.
 The method, started in \cite{ACL BIF CHAOS}, consists of using symmetry to obtain  basic dynamics and then choosing special symmetry-breaking terms with a simple form that preserve the required properties, and that introduce some desired behaviour.}

Some dynamical properties  {in this article} follow applying results in the literature. We also describe the transition between the different types of dynamics. As far as we know, both the construction and the study of the transition between dynamics are new.
 {This} article  {is also}
 part of a  {program of} systematic study
of the dynamics near networks with rotating nodes{, together with \cite{LR}}. 
 {The aim of the present article is the study of}
the non-wandering set that appear{s}  in the flow whereas 
 {the focus of \cite{LR} is}
 the bifurcation diagram. 
 Along this discussion we 
 compare 
 our results to
 what is known for other models in the literature. 

\subsection{Shilnikov homoclinic cycles}
For three-dimensional flows, one  {unstable}
homoclinic cycle of Shilnikov type is enough to predict the existence of a countable infinity of suspended {spiralling} horseshoes \cite{Shilnikov65, Shilnikov68}.  Their existence does not require breaking the homoclinic connection as in the case of saddles whose linearisation has only real eigenvalues.   Under a dissipative condition, strange attractors similar to $\Lambda$ {of Theorem \ref{TeoremaSimetrico}} appear in the neighbourhood of a Shilnikov orbit \cite{Homburg}. These attractors are character{ise}d by the lack of uniform hyperbolicity, by the existence of a trajectory with positive Lyapunov exponent and by the existence of an open set in their basin of attraction. Many of the results in this paper are typical  {of the behaviour} in the neighbourhood of a single homoclinic cycle to a saddle-focus, although the dynamics in our case is 
 {richer, due to}
 the coexistence of different types of phenomena for the same vector field.

\subsection{Bykov cycles}
 {The example presented in this article}
has an interesting  {relation}
with the works of Glendinning and Sparrow \cite{GS1, GS2} about homoclinic cycles and T-points.  {These} authors studied the existence of multiround heteroclinic cycles in a two-dimensional  {parameter}
diagram. They found a logarithmic spiral of homoclinic cycles and more complicated Bykov cycles. 
 {In their article, these}
authors do not break the one-dimensional heteroclinic connection. Following the same lines, Bykov \cite{Bykov93, Bykov99} studied the codimension 2 case of the same kind of cycle for general systems and found infinitely many periodic solutions and homoclinic cycles of Shilnikov type.
Our results  {differ in some parts from those of Bykov}
but the similarities are consistent. More precisely, if we allow $\lambda_2$ to be a two-dimensional parameter, then for $\lambda_1$ fixed, we would find the same type of logarithmic spirals corresponding to homoclinic cycles. Another  {difference}
is that Theorem \ref{TeoremaSimetrico} and Corollary \ref{Switching}  do 
not depend on the ratio of the eigenvalues of the rotating nodes, as  {is the case} in \cite{Bykov99}.


\subsection{Robustly transitive sets}
The spiralling set $\Lambda$ of Theorem \ref{TeoremaSimetrico} is remarkably different from the robustly transitive sets described by Ara\'ujo and Pac\'ifico \cite{AP} and Morales \emph{el al} \cite{MPP}. This is easily seen using assertion~\emph{\eqref{wHomo}} of Theorem~\ref{ThHorseshoes} to obtain, for arbitrarily small $\lambda_2$, a homoclinic cycle to $\ww$ of Shilnikov type with the expanding eigenvalue greater than the contracting one. This induces the appearance of an arbitrarily large number of attracting  {and}
repelling periodic solutions. 
If  $\Lambda$ were robustly transitive, this would imply that $\XX_1$ could not be $C^1$--approximated by vector fields exhibiting either attracting or repelling sets, which would be a contradiction. Although the spiralling set $\Lambda$ is not robustly transitive, by Theorem~\ref{TeoremaRede} it is persistent under $\ZZ_2$--symmetric perturbations. 

\subsection{Contracting Lorenz models}
The study of these spiralling sets has not attracted as much attention as the Lorenz attractor because, in general, it is difficult to understand the topology associated to non-real eigenvalues.  The fact that these wild spiral sets contain equilibria makes them similar to Lorenz-like attractors described by Rovella \cite{Rovella}. Both examples show suspended horseshoes, coexisting with equilibria and periodic solutions. 

In our example, suspended horseshoes are present for an open set of parameters whereas in Rovella's example they  appear for a set of parameters with positive Lebesgue measure. Moreover, in our example, the suspended horseshoes coexist with heteroclinic tangencies and these in turn give rise to attracting periodic orbits, whose basins of attraction may be situated in the gaps of the horseshoes accumulating on the repelling original Bykov cycle \cite{Rodrigues3}. The attractor splits into infinitely many components, whereas Rovella's example satisfies Axiom A for a set of parameters with positive Lebesgue measure.

\subsection{Quasistochastic attractors}
The features of our example fit in the properties of the \emph{quasistochastic attractors} studied by Gonchenko \emph{et al} \cite{GST}. The tangencies proved in Theorem \ref{tangencies} give rise to attracting periodic solutions which coexist with the hyperbolic set $\mathcal{H}$; the basins of  {attraction}
of some sinks 
lie in the gaps of $\mathcal{H}$. 
 {The transitive non-isolated set $\Lambda$ surrounds  closed trajectories of different indices, in sharp contrast to what is expected of attractors that are either hyperbolic or Lorenz-like.}

 {Yet another feature that differs from}
hyperbolic sets is that $\Lambda$ should not possess the property of \emph{self-similarity}. There may exist infinitely many time scales on which behavior of the system is qualitatively different. We would like to stress that our attracting limit set $\Lambda$, composed by unstable orbits, may 
 {support}
 proper invariant measures (Sinai-Bowen-Ruelle measures) and, therefore, it may be studied in ergodic terms.

\medbreak
A generic perturbation of the partially symmetric vector field $\XX_1$ still has some  spiralling structure as reported in \cite{ACT1, ACT2}. 
It contains a finite sequence of topological horseshoes semiconjugate to full shifts over an alphabet with a finite number of symbols, instead of an infinite sequence. We believe that our example should be  { explored further} 
 because it puts together
 different phenomena. 
 {The} two different types of chaos reported in \cite{Rossler} have been observed in Theorem \ref{ThHorseshoes} and Figure \ref{perturbation1_horseshoes}. The behaviour near these specific networks can be lifted to larger networks as those reported in \cite{ALR, GST, Rodrigues2, RLA}.

\subsection{Final Remarks}
The properties of the family $\XX$ that we have described depend strongly on the orientation around 
the heteroclinic connection $[\vv \to \ww]$.
If the rotations inside $V$ and inside $W$ had opposite orientations, 
for some trajectories these two rotations in $V$ and $W$ would cancel out. The Bykov cycle would no  longer be
repelling. 
Newhouse phenomena are dense in a Bykov cycle of this type, and hence infinitely many strange attractors emerge. A systematic study of this case is in preparation using the concept of \emph{chirality}.

A lot more needs to be done before we understand well the dynamics of systems close to symmetry involving rotating nodes. 
Besides the interest of the study of the dynamics arising in generic unfoldings of an attracting heteroclinic network, its analysis is important because in the fully non-equivariant case the explicit analysis of the first return map seems intractable. Although the symmetry is not essential, in this article we were able to predict qualitative features  because the non-symmetric dynamics is close to symmetry.
\bigbreak
\textbf{Acknowledgements:} The authors would like to express their gratitude to Manuela Aguiar (University of Porto) for helpful discussions at the beginning
of this work. Also special thanks to Maria Lu\'isa Castro (University of Porto) for the numerical simulations in \emph{Matlab} shown in figure \ref{simulation1(geom)}. 

\appendix
\section{Transversality of invariant manifolds}\label{AppendixTransversality}
\subsection{Melnikov method revisited}
\label{Melnikov}

Melnikov \cite{Melnikov1} studied a method to find the transverse intersection of the the invariant manifolds for a time periodic perturbation of a homoclinic cycle. The pioneer idea of Melnikov is to make use of the globally computable solutions of the unperturbed system the computation of perturbed solutions. 
We start this appendix with a short description of this theory 
applied to saddle-connections. For a detailed proof for the homoclinic case, see Guckenheimer and Holmes \cite[section 4.5]{GH}.

If $X \in \RR^2, t \in \RR$ and $0<\varepsilon <\!\!<1$, consider the planar system:
\begin{equation}
\label{diffeo}
\dot{X}=f(X)+\varepsilon g(X,t)
\end{equation}
such that:
\begin{itemize}
\item there exists $T>0$ such that $g(X,t)=g(X, t+T)$ for all $t$ \emph{ie} $g$ is $T$-periodic;
\item for $\varepsilon=0$, the flow has a heteroclinic connection $\Gamma_0$ associated to two hyperbolic points $p_0$ and $p_1$;
\item the unstable manifold of $p_0$ coincides with the stable manifold of $p_1$.
\end{itemize}

Associated to the system (\ref{diffeo}), taking $\textbf{S}^1 \cong \RR / T$, we define the suspended system:
$$
\dot{X}=f(X)+\varepsilon g(X, \theta), \quad \dot{\theta}=1, \quad (X, \theta) \in \RR^2 \times S^1.
$$

With the above assumptions, $\{p_0\} \times \textbf{S}^1$ and $\{p_1\} \times \textbf{S}^1$ are hyperbolic periodic solutions for the suspended flow, whose invariant manifolds $\Gamma_0 \times \textbf{S}^1$ coincide. Since the limit cycles are hyperbolic, their hyperbolic continuation is well defined for $\varepsilon \neq 0$; hereafter we denote them by $\{p_0^\varepsilon \} \times \textbf{S}^1$ and $\{p_1^\varepsilon\} \times \textbf{S}^1$. Under general conditions (without symmetry for example), the heteroclinic connection $\Gamma_0 \times \textbf{S}^1$ is not preserved and for a non-empty open set in the parameter space, the invariant manifolds meet  transversely. Their intersection consists of a finite number of trajectories.

On a transverse cross section to  the perturbed system, the splitting of the stable and unstable manifolds
is measured by  the \emph{Melnikov function}:
\begin{equation}
\label{Meln1}
M(t_0)=\int^{+\infty}_{-\infty} f(q_0(t))\wedge g(q_0(t),t+t_0).\exp  \left(- \int^{t}_0 tr Df(q_0(s)) ds \right)dt,
\end{equation}
where $q_0(t)$ is the parametrisation of the solution of the unperturbed system (\ref{diffeo}) starting at $t_0=0$ on $\Gamma_0$. Recall that the \emph{wedge product} in $\RR^2$ of two vectors $(u_1, u_2)$ and $(v_1, v_2)$ is simply given by $u_1  v_2 - u_2 v_1$. 
When $f$ is  hamiltonian, then $tr Df(q_0(t))=0$ and {the expression for the }Melnikov function is simpler. 
The main result we will use is the following:

\begin{theorem}[Bertozzi \cite{Bertozzi}, Melnikov \cite{Melnikov1}, adapted]
Under the above conditions, and for $\varepsilon>0$ sufficiently small, if $M(t_0)$ has simple zeros, then $W^u(p_0^\varepsilon)$ and  $W^s(p_1^\varepsilon)$ intersect  transversely .
\end{theorem}

This result is important because it allows to prove the existence of transverse (and hence robust) homo and heteroclinic connections as we proceed to do.

\subsection{Proof of transversality}\label{Transversality}
\begin{theorem}
\label{transv proof example}
If $\lambda_1 \neq 0$ and $\lambda_2=0$, the two-dimensional invariant manifolds of the equilibria $\ww$ and $\vv$ of $\XX(X,\lambda_1 \lambda_2)$ intersect  transversely  in $\EU^3$ along one-dimensional orbits.
\end{theorem}

\begin{proof}In spherical coordinates
$$
x_1=r \sin \phi\ \sin \theta \cos \varphi \qquad
x_2=r \sin \phi\ \sin \theta \sin \varphi \qquad
x_3=r \cos \phi \sin \theta \qquad
x_4=r \cos \theta
$$
equations (\ref{example}) restricted to $\EU^3$ can be written as:
\begin{equation}
\begin{array}{l}
\dot{\theta} = \alpha_1  \sin \theta \cos (2\phi) + \frac{\alpha_2}{2}\sin(2\theta) + \frac{\lambda_{1} }{2}\sin^2(\phi)\sin^2(\theta)\cos(\phi)\sin(2\varphi)\\
\dot{\phi}=-\alpha_1  \cos (\theta) \sin (2\phi) -\frac{\lambda_{1} }{4} \sin^3(\phi)\sin(2\theta)\sin(2\varphi)\\
\dot{\varphi}=1
\label{system spherical coordinates}
\end{array}
\end{equation}

From the equation $\dot{\varphi}=1$, we get $\varphi(t)=t$, $t \in \RR$ and (\ref{system spherical coordinates}) is reduced to a three-dimensional non-autonomous differential equation
of the form:
$$
\begin{array}{l}
\dot{\theta}=f_1(\theta, \phi)+\lambda_{1} g_1(\theta, \phi, t)\\
\dot{\phi}=f_2(\theta, \phi)+\lambda_{1} g_2(\theta, \phi, t)\\
\end{array}
$$
where
$$
\begin{array}{ll}
f_1(\theta, \phi) =  \alpha_1  \sin \theta \cos (2\phi) + \frac{\alpha_2}{2}\sin(2\theta)&
f_2(\theta, \phi) = -\alpha_1  \cos (\theta) \sin (2\phi)\\
g_1(\theta, \phi, t) = \frac{1}{2}\sin^2(\phi)\sin^2(\theta)\cos(\phi)\sin(2t)\qquad&
g_2(\theta, \phi, t) = -\frac{1}{4} \sin^3(\phi)\sin(2\theta)\sin(2t) \ .
\end{array}
$$
The maps $g_1$ and $g_2$ are periodic in $t$ of period $\pi$. For the Melnikov function $M(t_0)$ defined in \eqref{Meln1} we write
$f=(f_1,f_2)$ and $g=(g_1,g_2)$. The parametrisation $q_0(t)$ of the connection $[\textbf{w} \rightarrow \textbf{v}]$ in the unperturbed system, $\lambda_1=0$, is defined by  $\phi=\frac{\pi}{2}+k\pi$, $k \in \{0,1\}$. Thus, in the unperturbed system, the connections $[\textbf{w} \rightarrow \textbf{v}]$ are parametrised by:
$$
q_0^1(t)=\left(\theta(t), \frac{\pi}{2}\right) \text{   and   } q_0^2(t)=\left(\theta(t), \frac{3\pi}{2}\right).
$$
Therefore, for $k \in \{0,1\}$, we have:
$$
\begin{array}{ll}
f_1(q_0^i(t)) =  (-1)^k \alpha_1  \sin (\theta(t)) + \frac{\alpha_2}{2}\sin(2\theta(t))\qquad&
f_2(q_0^i(t)) = 0\\
g_1(q_0^i(t), t+t_0) = 0 &
g_2(q_0^i(t), t+t_0) = (-1)^{k+1} \frac{1}{4} \sin(2\theta(t))\sin(2(t+t_0)) \ .\\
\end{array}
$$

To see that  the integral $M(t_0)$ converges, note that the  exterior product in the definition of the Melnikov function $M(t_0)$ is bounded since it is given by:
$$
\begin{array}{l}
f(q_0^i(t))\wedge g(q_0^i(t),t+t_0)=\\
=\left[ -\alpha_1  \sin (\theta(t)) + (-1)^{k+1}\frac{\alpha_2}{2}\sin(2\theta(t))\right] \frac{1}{4} \sin(2\theta(t))\sin(2(t+t_0))
\end{array}
$$
and
\begin{equation}
\label{traco}
tr Df(q_0(s))=\alpha_1  \cos (\theta(s))+\alpha_2 \cos (2\theta(s))
\end{equation}
hence the Melnikov integral  $M(t_0)$ does not depend on $\lambda_1$.
Since it has been shown in Aguiar \emph{et al} \cite[Lemma 16]{ACL BIF CHAOS} that for any $r>0$, the integral
$$
\int^{+\infty}_{-\infty} \exp{\left(  -\int^{t}_0 \alpha_2 r \cos (\theta(s))+\alpha_3 r^2 \cos (2\theta(s))ds \right)} dt
$$
converges, the convergence of $M(t_0)$ follows. It remaisn to prove that the roots of $M(t_0)$ exist and are simple, completing the proof of Theorem \ref{transv proof example}. This is the content of next lemma whose proof is similar to results in \cite{ACL BIF CHAOS}.
\end{proof}

\begin{lemma} 
The Melnikov integral $M(t_0)$ has simple roots.
\end{lemma}

\begin{proof}
Using the expression (\ref{traco}) and the expression of the $\sin$ of the sum, we may write $M(t_0)$ in the form:
$$
M(t_0)= \cos(2t_0)\int_{-\infty}^{+\infty} [\sin(2t)] E(t) dt+ \sin(2t_0)\int_{-\infty}^{+\infty} [\cos(2t)] E(t) dt
$$
where
$$
E(t)=\left[-\alpha_1\sin(\theta(t))+(-1)^{k+1}\frac{\alpha_2}{2}\sin(2\theta(t))\right]\left[\frac{1}{4}\sin(2\theta(t))\right]\exp(-trDf(q_0(s))).
$$
Suppose that $t_0$ is a non-simple zero of $M(t_0)$. Since $t_0$ is a zero, then:
\begin{equation}
\label{(4.11)}
\cos(2t_0)\int_{-\infty}^{+\infty} [\sin(2t)] E(t) dt+ \sin(2t_0)\int_{-\infty}^{+\infty} [\cos(2t)] E(t) dt=0,
\end{equation}
or equivalently
$$
\tan(2t_0)=-\frac{\int_{-\infty}^{+\infty} \sin(2t) E(t) dt}{\int_{-\infty}^{+\infty} \cos(2t) E(t) dt}.
$$
Since $t_0$ is non-simple, differentiating (\ref{(4.11)}) with respect to $t_0$ we must have:
$$
-\sin(2t_0)\int_{-\infty}^{+\infty} [\sin(2t)] E(t) dt+ \cos(2t_0)\int_{-\infty}^{+\infty} [\cos(2t)] E(t) dt=0,
$$
and thus:
$$\tan(2t_0)=\frac{\int_{-\infty}^{+\infty} \cos(2t) E(t) dt}{\int_{-\infty}^{+\infty} \sin(2t) E(t) dt},$$
which is a contradiction. It remains to show that $M(t_0)$ has at least a zero. For this purpose, write:
$$
\rho \exp(-i\varsigma)=\int_{-\infty}^{+\infty} \exp(-2it) E(t) \qquad \text{whence} \qquad M(t_0)= \rho Re \left(\exp(i(2t_0-\varsigma ))\right).
$$
Thus $M(t_0)$ has zeros at $t_0=\frac{1}{4}(\pi + 2 \varsigma+2n\pi), n\in \ZZ$.
\end{proof}

\section{Symmetry breaking perturbations}\label{AppendixTable}

List of homogeneous polynomial vector fields of degree 3, tangent to $\EU^3$ and their symmetries in
$\textbf{SO(2)} \times \ZZ_2 (\gamma_2)$, that may be used for perturbations as in section \ref{SymmetryAlong}.
Adapted from Aguiar~\cite{tesemanela}.
\medbreak

\noindent Perturbing terms with $\textbf{SO(2)} \times \ZZ_2 (\gamma_2)$-symmetry:
$$
(x_2x_4^2, -x_1x_4^2, 0,0)
\qquad
(x_2x_3^2, -x_1x_3^2, 0,0)
\qquad
(0,0,x_3x_4^2, -x_3^1x_4)
$$
\bigbreak\noindent
Perturbing terms with $\textbf{SO(2)}$-symmetry, not $ \ZZ_2 (\gamma_2)$-symmetric:
$$
(0,0, x_4^3, -x_3x_4^2)\qquad
(x_2x_3x_4, -x_1x_3x_4, 0, 0)\qquad
(0,0, x_3^2x_4, -x_3^3)
$$
\bigbreak\noindent
Perturbing terms with $ \ZZ_2 (\gamma_1) \times \ZZ_2 (\gamma_2)$-symmetry, not $\textbf{SO(2)}$-symmetric:
\begin{center}
\begin{tabular}{llll}
$(x_1^2 x_2, -x_1^3,0,0)$ & $
(x_2^3, -x_1x_2^2, 0,0)$ & $
(0, x_1x_4^2, 0, -x_1x_2x_4)$ & $
(x_1x_4^2, 0,0,-x_1^2x_4)$
\\
$(x_2x_3^2, 0, -x_1x_2x_3,0)$ & $
(0, x_1x_3^2, -x_1x_2x_3,0)$ & $
(0,x_2x_3^2, -x_2^2x_3,0)$ & $
(x_1x_2^2, -x_1^2x_2, 0,0)$
\\
$(x_2x_4^2, 0, 0, -x_1x_2x_4)$&
$(0,x_2x_4^2, 0, -x_2^2x_4)$&
$(x_1x_3^2, 0, -x_1^2x_3, 0)$&
{}
\end{tabular}
\end{center}
\bigbreak\noindent
Perturbing terms with $ \ZZ_2 (\gamma_1)$-symmetry, not $\textbf{SO(2)}$-symmetric nor $\ZZ_2 (\gamma_2)$-symmetric:
\begin{center}
\begin{tabular}{llll}
$(0,0,x_1x_2x_4, -x_1x_2x_3)$&
$(0,x_2x_3x_4, 0, -x_2^2x_3)$&
$(x_1x_2x_3, -x_1^2x_3, 0,0)$&
$(0,0,x_1^2x_4, -x_1^2x_3)$
\\
$(0,x_1x_3x_4,-x_1x_2x_4, 0)$&
$(0,x_1x_3x_4, 0, -x_1x_2x_3)$&
$(0, x_2x_3x_4, -x_2^2x_4, 0)$&
$(x_2x_3x_4, 0, -x_1x_2x_4, 0)$ 
\\
 $(x_2x_3x_4, 0,0, -x_1x_2x_3)$&
 $(0,0, x_2^2x_4, -x_2^2x_3)$&
 $(x_1x_3x_4, 0,-x_1^2x_4, 0)$&
 $(x_1x_3x_4, 0,0,-x_1^2x_3)$ 
\end{tabular}
\end{center}
\bigbreak\noindent
Perturbing terms with $ \ZZ_2 (\gamma_2)$-symmetry, not $\textbf{SO(2)}$-symmetric nor $\ZZ_2 (\gamma_1)$-symmetric:
\begin{center}
\begin{tabular}{llll}
$(x_3^2x_4, 0, -x_1x_3x_4, 0)$&
$(0, x_1^2x_4, 0, -x_1^2x_2)$&
$(x_2^2x_4, -x_1x_2x_4, 0,0)$&
$(0,x_2^2x_4, 0, -x_2^3)$
\\
$(0,x_3^2x_4,0, -x_2x_3^2)$&
$(x_2^2x_4, 0, 0, -x_1x_2^2)$ &
$(x_3^2x_4, 0,0, -x_1x_3^2)$&
$(0,x_3^2x_4, -x_2x_3x_4,0)$
\\
$(x_1^2x_4, 0,0,-x_1^3)$&
$(x_1x_2x_4, -x_1^2x_4, 0,0)$&
$(0, x_1x_2x_4, 0, -x_1x_2^2)$&
$(0,0,0,x_1x_3x_4, -x_1x_3^2)$
\\
$(x_4^3, 0,0, -x_1x_4^2)$&
$(x_1x_2x_4, 0,0,-x_1^2x_2)$&
$(0,0,x_2x_3x_4, -x_2x_3^2)$&
$(0, x_4^3, 0, -x_2x_4^2)$
\end{tabular}
\end{center}

\bigbreak\noindent
Perturbing terms without any of the symmetries above:
\begin{center}
\begin{tabular}{llll}
$(x_1x_2x_3, 0, -x_1^2x_2,0)$&
$(x_3x_4^2, 0, 0, -x_1x_3x_4)$&
$(0, x_3^3, -x_2x_3^2,0)$&
$(x_1^2x_3, 0,-x_1^3, 0)$
\\
$(0, x_1^2x_3, -x_1^2x_2,0)$&
$(x_2^2x_3, -x_1x_2x_3, 0,0)$&
$(0,0,x_1x_4^2, -x_1x_3x_4)$&
$(0, x_3x_4^2, 0, -x_2x_3x_4)$
\\
$(x_2^2x_3, 0,-x_1x_2^2, 0)$&
$(0,0,x_2x_4^2, -x_2x_3x_4)$&
$(0,x_1x_2x_3, -x_1x_2^2, 0)$& 
$(x_3^3, 0, -x_1x_3^2, 0)$
\\
$(x_3x_4^2,0, -x_1x_4^2, 0)$&
$(0, x_2^2x_3,-x_2^3, 0)$&$(0,x_3x_4^2, -x_2x_4^2, 0)$&
\end{tabular}
\end{center}

\end{document}